\documentclass{amsproc}

\usepackage{amsmath,amssymb}
\newtheorem{theorem}{Theorem}[section]
\newtheorem{lemm}[theorem]{Lemma}
\newtheorem{prop}[theorem]{Proposition}

\theoremstyle{definition}
\newtheorem{defi}[theorem]{Definition}

\newtheorem{coro}[theorem]{Corollary}
\theoremstyle{remark}
\newtheorem{remark}[theorem]{Remark}

\numberwithin{equation}{section}

\def\vn{\varepsilon}

\def\ot{\otimes}
\def\b{\overline}

\def\om{\omega}

\def\ze{\zeta}

\def\a{\alpha}

\def\b{\beta}

\newfont{\df}{eufm10}

\def\ot{\otimes}

\def\ot{\otimes}

\begin{document}

\title[Restricted two-parameter quantum group of type $G_2$ ]
{Convex PBW-type Lyndon Basis and Restricted Two-parameter Quantum
Group of Type $G_2$ }

\author[Hu]{Naihong Hu$^\star$}
\address{Department of Mathematics, East China Normal University,
Min Hang Campus, Dong Chuan Road 500, Shanghai 200241, PR China}
\email{nhhu@math.ecnu.edu.cn}

\author[Wang]{Xiuling Wang$^*$}
\address{School of Mathematical Science, Nankai University,
Tianjin  300071, PR China}\email{xiulingwang@nankai.edu.cn}
\thanks{$^\star$
N. Hu, supported in part by the NNSF (Grants 10431040, 10728102),
the PCSIRT from the MOE of China, the National/Shanghai Priority
Academic Discipline Programmes}
\thanks{$^*$ X. Wang, Corresponding Author, supported
by the Nankai Research-encouraging Fund for the PhD-Teachers and a
fund from the LPMC}

\today{}
\date{June 18, 2005}


\begin{abstract}
We construct finite-dimensional pointed Hopf algebras $\mathfrak
u_{r,s}(G_2)$ (i.e. restricted 2-parameter quantum groups) from the
2-parameter quantum group $U_{r,s}(G_2)$ defined
 in \cite{HS}, which turn out to be of
Drinfel'd doubles, where a crucial point is to give a detailed
combinatorial construction of the convex PBW-type Lyndon basis for
type $G_2$ in 2-parameter quantum version.  After furnishing
possible commutation relations among quantum root vectors, we show
that the restricted quantum groups are ribbon Hopf algebras under
certain conditions through determining their left and right
integrals. Besides these, we determine all of the Hopf algebra
isomorphisms of $\mathfrak u_{r,s}(G_2)$ in terms of the description
of the sets of its left (right) skew-primitive elements.
\end{abstract}

\maketitle
\section{Introduction}

Hopf algebra was first observed in algebraic topology by H. Hopf
early in 1941 and, as purely algebraic objects, has been developed
by many mathematicians and applied to other areas of mathematics
such as Lie theory, knot theory and combinatorics, etc. A
longstanding problem in the area is the full classification of the
finite-dimensional Hopf algebras. One of the very few general
classification results says that any cocommutative Hopf algebra over
the complex field $\mathbb C$ is a semidirect product of the
universal enveloping algebra of a Lie algebra and a group algebra,
which was known  as the Cartier-Kostant-Milnor-Moore theorem. Since
Kaplansky's ten conjectures \cite{K} proposed in 1975, they have
stimulated much research on Hopf algebras, and there have been a lot
of significant advances during the last two decades. A rich supply
of examples of noncommutative and noncocommutative Hopf algebras are
the Drinfel'd-Jimbo quantum groups $U_q(\mathfrak{g})$ (see
\cite{Dr}, $\mathfrak{g}$ is a semisimple Lie algebra) which were
found in the mid-eighties of the last century, and the
finite-dimensional small quantum groups $\mathfrak u_q(\mathfrak g)$
at $q=\epsilon$ a root of unity introduced by Lusztig \cite{L1}.
Until now the classification splits into two cases: the semisimple
case and the non-semisimple case. A good overview on classification
of semisimple Hopf algebras  is \cite{M1}. The classification of
non-semisimple Hopf algebras focuses on those of pointed Hopf
algebra over an algebraically closed field of characteristic 0
\cite{AS3,AS4}. Pointed Hopf algebras play an important role in
\cite{AS1,BDG,G}, where Kaplansky's 10th conjecture is refuted by
constructing infinitely many nonisomorphic Hopf algebras of a given
prime power dimension. Since finite-dimensional Hopf algebras are
far from being classified, it is meaningful to have various means of
constructing examples of finite-dimensional Hopf algebras (see for
instance \cite{AS1, AS2, BW3, H1, H2, HW1, HW2, L1, R, T}, etc.). To
our interest,  \cite{BW3, HW2} determined the structure of
restricted two-parameter quantum groups $\mathfrak u_{r,s}(\mathfrak
{sl}_n)$ and $\mathfrak u_{r,s}(\mathfrak {so}_{2n+1})$,
respectively, when both parameters $r, s$ are roots of unity, which
are new finite-dimensional pointed Hopf algebras and have new ribbon
elements under some conditions. Besides these, this will act as a
starting point for further studying the representation theory of the
two-parameter quantum groups at roots of unity as in the
one-parameter setting (cf. \cite{DK, L1} etc.). The goal of this
article is to solve the same problems for the type $G_2$ case.

\medskip

Two-parameter or multiparameter quantum groups since the work of
Drinfel'd \cite{Dr}  and Jimbo  \cite{Jim} were  focused on
quantized function algebras and quantum enveloping algebras mainly
for type $A$ cases in 1990's. In 2001, from the motivation of
down-up algebras approach \cite{Be}, Benkart-Witherspoon in
\cite{BW1} re-obtained Takeuchi's definition of two-parameter
quantum groups of $\mathfrak{gl}_n$ and $\mathfrak{sl}_n$. Since
then, a systematic study for the two-parameter quantum groups of
type $A$ has been developed by Benkart, Witherspoon, and their
cooperators Kang, Lee (see \cite{BKL1, BKL2, BW2, BW3}). In 2004,
Bergeron-Gao-Hu defined the two-parameter quantum groups
$U_{r,s}(\mathfrak g)$ (in the sense of Benkart-Witherspoon) for
$\mathfrak g=\mathfrak {so}_{2n+1}$, $\mathfrak {so}_{2n}$ and
$\mathfrak {sp}_{2n}$ in \cite{BGH1}, which are realized as
Drinfel'd doubles, and described weight modules in the category
$\mathcal O$ when $rs^{-1}$ is not a root of unity (see
\cite{BGH2}). Afterwards, Hu and his cooperators continued to
develop the corresponding theory for exceptional types $G, E$ and
the affine cases in \cite{HS, BH,HRZ,HZ}, etc.

\medskip

It should be mentioned that  one cannot write out conveniently the
convex PBW-type bases for the two-parameter quantum groups in terms
of Lusztig's braid group actions as the typical fashion in the
one-parameter case (see \cite{Ka, L2, L3}), which is witnessed by
Theorem 3.1 \cite{BGH1} in the study of Lusztig's symmetry. This is
one of difficulties encountered in the two-parameter setting, while
the combinatorial construction of the PBW-type bases in the quantum
setup is a rather nontrivial matter, which definitively depends not
only on the choice of a convex ordering on a positive root system
(\cite{B, R2}), but also on the adding manner of the $\bf
q$-bracketings (see \cite{R2, K1, K2, BH, HRZ}, etc.) on good Lyndon
words (cf. \cite{LR, R2}). Note that the construction of convex
PBW-type bases in two-parameter quantum cases has been given for
type $A$ in \cite{BKL1}, types $E$ in \cite{BH}, and for the low
rank cases of types $B$, $C$, $D$ in \cite{H3}, for type $B$ for
arbitrary rank in \cite{HW2}.  Motivated by \cite{H3, HW2, LR, R2},
the first result of the article is to present an explicit
construction for the convex PBW-type Lyndon bases of type $G_2$.
This is a prerequisite of the whole discussions later.

\medskip
In  this paper, we  construct a family of finite-dimensional pointed
Hopf algebra $\mathfrak{u}_{r,s}(G_2)$ of dimension $\ell^{16}$ as a
quotient of $U_{r,s}(G_2)$ by a Hopf ideal $\mathcal{I}$, which is
generated by certain homogeneous central elements of degree $\ell$,
where $r$ is a primitive $d$th root of unity, $s$ is a primitive
$d'$th root of unity and $\ell$ is the least common multiple of $d$
and $d'$. We will assume that the ground field $\mathbb{K}$ contains
a primitive $\ell $th root of unity, and show that the restricted
quantum groups $\mathfrak{u}_{r,s}(G_2)$ is a ribbon Hopf algebra.
Owing to the complexity of Lyndon bases in nonsimply-laced Dynkin
diagram cases, our treatments in type $G_2$ are complicated.

\medskip
The article is organized as follows. In Section 2, we recall the
definition of the two-parameter quantum groups of type $G_2$ from
\cite{HS}, and some basics about the structure. In particular, we
present a direct construction of the convex PBW-type Lyndon bases
for $U_{r,s}(G_2)$. In Section 3, we at first contribute more
efforts to make the possible commutation relations
 clearly, and then determine those central elements of degree $\ell$ in the case when
taking both parameters $r$ and $s$ to be roots of unity, which
generate a Hopf ideal. These enable us to further derive the
restricted two-parameter quantum groups $\mathfrak u_{r,s}(G_2)$ as
required. In Section 4, we show that these Hopf algebras are
pointed, and determine all of the Hopf algebra isomorphisms of
$\mathfrak u_{r,s}(G_2)$ in terms of the description of the set of
left (right) skew-primitive elements of it, while in the type $A$
case, \cite{BW3} missed some important left (right) skew-primitive
elements (see (3.6) \& (3.7) {\it loc cit}) which led to some
interesting families of isomorphisms undiscovered. In Section 5, we
show that these Hopf algebras are Drinfel'd doubles of a certain
(Borel-type) Hopf subalgebra $\mathfrak{b}$. In Section 6, we
determine the left and right integrals of
 $\mathfrak{b}$ and use them in combining with a result of Kauffmann-Radford
\cite{KR} to give a characterization of $\mathfrak{u}_{r,s}(G_2)$ to
be a ribbon Hopf algebra.

\section{ $U_{r,s}(G_2)$ and the convex PBW-type Lyndon basis} 
\medskip

We begin by recalling the definition of two-parameter quantum group
of type $G_2$, which was introduced by Hu-Shi \cite{HS}.

\subsection{Two-parameter Quantum Group $U_{r,s}(G_2)$}

Let ${\mathbb K}={\mathbb Q}(r,s)$ be a field of rational functions
with two indeterminates $r$, $s$ ($r^3\neq s^3, r^4\neq s^4$).  Let
$\Phi$ be a finite root system of $G_2$ with $\Pi$ a base of simple
roots, which is a subset of a Euclidean space $E = {\mathbb R}^3$
with an inner product $(\,,\,)$. Let $\epsilon _{1},\,\epsilon
_{2},\,\epsilon_{3}$ denote an orthonormal basis of $E$, then $\Pi =
\{\alpha_{1} = \epsilon_{1}-\epsilon_{2},\; \alpha_{2} =
\epsilon_{2}+\epsilon_{3}-2\epsilon_{1}\}$ and $\Phi = \pm
\{\alpha_1,\alpha_2,\alpha_2+\alpha_1,\alpha_2+2\alpha_1,\alpha_2+3\alpha_1,2\alpha_2+3\alpha_1\}$.
In this case, we set $\displaystyle r_1 =
r^{\frac{(\alpha_1,\,\alpha_1)}{2}} = r,\; r_2 =
r^{\frac{(\alpha_2,\,\alpha_2)}{2}} = r^3$ and $s_1 =
s^{\frac{(\alpha_1,\,\alpha_1)}{2}} = s,\; s_2 =
s^{\frac{(\alpha_2,\,\alpha_2)}{2}} = s^3$.

\begin{defi}
Let $U=U_{r, s}(G_2)$ be the associative algebra over ${\mathbb
Q}(r,s)$ generated by symbols $e_i,\;f_i,\;\omega_i^{\pm 1},\;
\omega_i'^{\pm 1} \;(i=1, 2)$ subject to the relations
$(G1)$---$(G6)$: \vskip0.2cm

\noindent $(G1)$ \ $[\,\omega_i^{\pm 1}, \omega_j^{\pm
1}\,]=[\,\omega_i^{\pm 1}, \omega_j'^{\pm 1}\,]=[\,\omega_i'^{\pm
1}, \omega_j'^{\pm 1}\,]=0, \quad \omega_i\omega_i^{-1} =1=
\omega_j'\omega_j'^{-1}$.

\smallskip

\noindent $(G2)$
\hspace*{\fill}$\omega_{1}\,e_{1}\,\omega_{1}^{-1} =
(rs^{-1})\,e_{1},\qquad\quad\ \;
\omega_{1}\,f_{1}\,\omega_{1}^{-1} =
(r^{-1}s)\,f_{1}$,\hspace*{\fill}

 \hspace*{\fill}$\omega_{1}\,e_{2}\,\omega_{1}^{-1} =
s^{3}\,e_{2},\qquad\qquad\quad \omega_{1}\,f_{2}\,\omega_{1}^{-1}
= s^{-3}\,f_{2}$,\hspace*{\fill}

\hspace*{\fill}$\omega_{2}\,e_{1}\,\omega_{2}^{-1} =
r^{-3}\,e_{1},\qquad\quad\ \;\, \omega_{2}\,f_{1}\,\omega_{2}^{-1}
= r^{3}\,f_{1}$,\hspace*{\fill}

\hspace*{\fill}$\omega_{2}\,e_{2}\,\omega_{2}^{-1} =
(r^{3}s^{-3})\,e_{2},\qquad\qquad\
\omega_{2}\,f_{2}\,\omega_{2}^{-1} =
(r^{-3}s^{3})\,f_{2}$.\hspace*{\fill}

\smallskip

\noindent $(G3)$
\hspace*{\fill}$\omega_{1}'\,e_{1}\,\omega_{1}'^{-1} =
(r^{-1}s)\,e_{1},\qquad\quad\ \
\omega_{1}'\,f_{1}\,\omega_{1}'^{-1} =
(rs^{-1})\,f_{1}$,\hspace*{\fill}

 \hspace*{\fill}$\omega_{1}'\,e_{2}\,\omega_{1}'^{-1} =
r^{3}\,e_{2},\qquad\qquad\quad
\omega_{1}'\,f_{2}\,\omega_{1}'^{-1} =
r^{-3}\,f_{2}$,\hspace*{\fill}

\hspace*{\fill}$\omega_{2}'\,e_{1}\,\omega_{2}'^{-1} =
s^{-3}\,e_{1},\qquad\quad\ \;\,
\omega_{2}'\,f_{1}\,\omega_{2}'^{-1} =
s^{3}\,f_{1}$,\hspace*{\fill}

\hspace*{\fill}$\omega_{2}'\,e_{2}\,\omega_{2}'^{-1} =
(r^{-3}s^{3})\,e_{2},\qquad\qquad\
\omega_{2}'\,f_{2}\,\omega_{2}'^{-1} =
(r^{3}s^{-3})\,f_{2}$.\hspace*{\fill}

\smallskip

\noindent $(G4)$ \ For $1\le i,\, j\le 2$, we have
$$
[\,e_i, f_j\,]=\delta_{ij}\frac{\om_i-\om_i'}{r_i-s_i}.
$$

\noindent $(G5)$ \, We have the following $(r,\,s)$-Serre
relations:
\begin{gather*}
e_2^{2}e_1 - (r^{-3} + s^{-3})\,e_{2}e_{1}e_{2} +
(rs)^{-3}\,e_{1}e_2^2 = 0,\tag*{$(G5)_1$}\\
\begin{split}
e_1^{4}e_2 - (r +& s)(r^2 + s^2)\,e_1^{3}e_{2}e_1 + rs(r^2 +
s^2)(r^2
+ rs + s^2)\,e_{1}^{2}e_{2}e_1^2\\
-& (rs)^3(r + s)(r^{2} + s^2)\,e_{1}e_{2}e_1^3+\, (rs)^6e_{2}e_1^4
= 0.
\end{split}\tag*{$(G5)_2$}
\end{gather*}

\noindent $(G6)$ \,  We have the following $(r,\,s)$-Serre
relations:
\begin{gather*}
f_{1}f_2^{2} - (r^{-3} + s^{-3})\,f_{2}f_{1}f_{2} +
(rs)^{-3}\,f_{2}^{2}f_1 = 0,\tag*{$(G6)_1$}\\
\begin{split}
f_{2}f_1^{4} - (r +& s)(r^2 + s^2)\,f_1f_{2}f_1^{3} + rs(r^2 +
s^2)(r^2 + rs + s^2)\,f_{1}^{2}f_{2}f_1^2\\
-& (rs)^3(r + s)(r^{2} + s^2)\,f_{1}^3f_{2}f_1 +\,
(rs)^6\,f_{1}^4f_2 = 0.
\end{split}\tag*{$(G6)_2$}
\end{gather*}
\end{defi}

The algebra $U_{r,s}(G_2)$ is a Hopf algebra, where the
$\om_i^{\pm1}, {\om_i'}^{\pm1}$  are group-like elements, and the
remaining Hopf structure is given by
\begin{gather*}
\Delta(e_i)=e_i\ot 1+\om_i\ot e_i, \qquad \Delta(f_i)=1\ot
f_i+f_i\ot \om_i',\\
\vn(\om_i^{\pm1})=\vn({\om_i'}^{\pm1})=1, \qquad
\vn(e_i)=\vn(f_i)=0,\\
S(\om_i^{\pm1})=\om_i^{\mp1}, \qquad
S({\om_i'}^{\pm1})={\om_i'}^{\mp1},\\
S(e_i)=-\om_i^{-1}e_i,\qquad S(f_i)=-f_i\,{\om_i'}^{-1}.
\end{gather*}

When $r = q=s^{-1}$, the Hopf algebra $U_{r,\,s}(G_2)$ modulo the
Hopf ideal generated by $\omega_i' - \omega_i^{-1} (i=1,2)$, is just
the quantum group $U_q(G_2)$ of Drinfel'd-Jimbo type.

In any Hopf algebra $\mathcal H$, there exist the left-adjoint and
the right-adjoint action defined by its Hopf algebra structure as
follows
$$
\text{ad}_{ l}\,a\,(b)=\sum_{(a)}a_{(1)}\,b\,S(a_{(2)}), \qquad
\text{ad}_{ r}\,a\,(b)=\sum_{(a)}S(a_{(1)})\,b\,a_{(2)},
$$
where $\Delta(a)=\sum_{(a)}a_{(1)}\ot a_{(2)}\in \mathcal H\otimes
\mathcal H$, for any $a$, $b\in \mathcal H$.

From the viewpoint of adjoint actions, the $(r,s)$-Serre relations
$(G5)$, $(G6)$ take the simplest forms:
\begin{gather*}
\bigl(\text{ad}_l\,e_i\bigr)^{1-a_{ij}}\,(e_j)=0,
\qquad\text{\it for any } \ i\ne j,\\
\bigl(\text{ad}_r\,f_i\bigr)^{1-a_{ij}}\,(f_j)=0, \qquad\text{\it
for any } \ i\ne j.
\end{gather*}

Let $U^+$ (resp. $U^-$) is the subalgebra of $U=U_{r,s}(G_2)$
generated by the elements $e_i$
 (resp. $f_i$) for $i=1,2$.  Moreover, let $\mathcal{B}$ (resp. $\mathcal{B'}$) denote
the Hopf subalgebra of $U_{r,\,s}(G_2)$, which is generated by
$e_j,\omega_j^{\pm 1}$ (resp. $f_j, \omega_j'^{\pm 1}$) with $j=1,
2$.

\begin{prop} There exists a unique
skew-dual pairing $\langle \,, \rangle: \,\mathcal{B'} \times
\mathcal{B} \longrightarrow \mathbb Q(r,\,s)$ of the Hopf
subalgebras $\mathcal{B}$ and $\mathcal{B'}$, such that
\begin{gather*}
\langle f_i,\,e_j \rangle = \delta_{ij} \frac{1}{s_i - r_i},
\qquad (1 \leq i,\,j \leq2),\\
\langle \omega_1',\,\omega_1 \rangle = rs^{-1},\quad
\langle \omega_1',\,\omega_2 \rangle = r^{-3}, \\
\langle \omega_2',\,\omega_1 \rangle = s^{3},\quad
\langle \omega_2',\,\omega_2 \rangle = r^{3}s^{-3},\\
\langle \omega_i'^{\pm 1},\,\omega_j^{-1} \rangle = \langle
\omega_i'^{\pm 1},\,\omega_j \rangle ^{-1} = \langle
\omega_i',\,\omega_j \rangle ^{\mp 1},\qquad (1 \leq i,\,j \leq 2),
\end{gather*}
and all other pairs of generators are $0$. Moreover, we have
$\langle S(a),\,S(b) \rangle = \langle a,\,b \rangle$ for $a\in
\mathcal{B}', b\in \mathcal{B}$. \hfill\qed
\end{prop}

\subsection{Convex PBW-type Lyndon basis}
Recall that a reduced expression of the longest element of Weyl
group $W$ for type $G_2$ taken as
$$
w_0=s_1s_2s_1s_2s_1s_2
$$
yields a convex
ordering of positive roots below:
$$\alpha_1,\quad 3\alpha_1+\alpha_2, \quad 2\alpha_1+\alpha_2,\quad 3\alpha_1+2\alpha_2,\quad \alpha_1+\alpha_2,\quad\alpha_2.
$$
Write the positive root system $\Phi^+ = \{\alpha_1,
3\alpha_1+\alpha_2, 2\alpha_1+\alpha_2, 3\alpha_1+2\alpha_2,
\alpha_1+\alpha_2, \alpha_2\}$.

Note that the above ordering also corresponds to the standard Lyndon
tree of type $G_2$: 

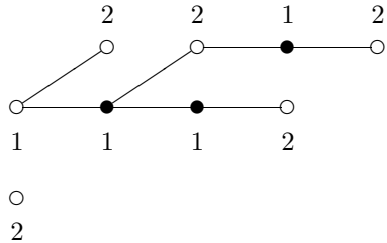
\unitlength 0.8mm 
\linethickness{0.1pt}
\ifx\plotpoint\undefined\newsavebox{\plotpoint}\fi 
\begin{center}
\begin{picture}(60,45)(0,-25)
\put(0,0){\circle{2}} \put(0,-15){\circle{2}}\put(15,0){\circle*{2}}
\put(30,0){\circle*{2}} \put(45,0){\circle{2}}
\put(15,10){\circle{2}}\put(30,10){\circle{2}}
\put(45,10){\circle*{2}} \put(60,10){\circle{2}}
\put(1,0){\line(1,0){14}} \put(16,0){\line(1,0){14}}
\put(31,0){\line(1,0){13}} \put(1,0.67){\line(3,2){13}}
\put(16,0.67){\line(3,2){13}}\put(31,10){\line(1,0){14}}
\put(46,10){\line(1,0){13}} \put(-1,-7){1} \put(14,-7){1}
\put(29,-7){1} \put(44,-7){2} \put(-1,-22){2} \put(14,14){2}
\put(29,14){2} \put(44,14){1}
 \put(59,14){2}
\end{picture}\\
\end{center}

With the above ordering and notation, we can make an inductive
definition of the quantum root vectors  $ E_\alpha$ in $U^+$ as
follows. Briefly, we set $E_{i}=E_{\a_i}, E_{12}=E_{\a_1+\a_2},
E_{112}=E_{2\a_1+\a_2}, E_{1112}=E_{3\a_1+\a_2},
E_{11212}=E_{3\a_1+2\a_2}$.

We define  inductively:
\begin{eqnarray}
E_{1}=e_1, \quad  E_{2}=e_2,\\
E_{12}=e_1e_2-s^3e_2e_1,
\\
E_{112}=e_1E_{12}-rs^2E_{12}e_1, \\
E_{1112}=e_1E_{112}-r^2sE_{112}e_1,\\
E_{11212}=E_{112}E_{12}-r^2sE_{12}E_{112}.
\end{eqnarray}
Then the defining relations for $U^{+}$ in  $(G5)$ can  be
reformulated as saying
\begin{eqnarray}
E_{12}e_{2}=r^3e_{2}E_{12},\\
e_1E_{1112}=r^3E_{1112}e_1.
\end{eqnarray}

\begin{remark}(i) The defining relations in (2.1)-(2.5) can be reformulated by
the left-adjoint action defined by its Hopf algebra structure, for
example, $E_{12}=\text{ad}_{ l}\,e_1\,(e_2)=e_1e_2-s^3e_2e_1.$

(ii) Note that the matrix $(\langle \omega_i',\,\omega_j
\rangle)_{2\times 2}$ is
$ \begin{pmatrix} rs^{-1} & r^{-3}\\
s^3& r^3s^{-3}
\end{pmatrix},$
set $ p_{ji}=\langle \omega_i',\omega_j\rangle.$ As for how to get
the quantum root vector $E_\alpha$ in $U^+$, we have to add
$(r,s)$-bracketings on each good Lyndon word obeying the defining
rule as $E_\gamma:=[E_\alpha,
E_\beta]_{\langle\omega_\beta',\omega_\alpha\rangle}=E_\alpha
E_\beta-\langle \omega_\beta',\omega_\alpha\rangle E_\beta E_\alpha$
for $\alpha, \gamma,\beta\in\Phi^+$ with $\alpha<\gamma<\beta$ in
the convex ordering, and $\gamma=\alpha+\beta$.
\end{remark}

Note that Kharchenko in \cite{K1} found a PBW-type  basis for a Hopf
algebra generated by an abelian group and a finite set of skew
primitive elements such that the adjoint action of the group on the
skew primitive generators is given by multiplication with a
character. Once we construct the quantum root vectors as above,
according to \cite{K1, R2}, we will have the following result.

\begin{theorem}\label{2.4}
 $\{E_2^{n_{1
}}E_{12}^{n_2}E_{11212}^{n_3}E_{112}^{n_4}E_{1112}^{n_5}E_1^{n_6}\mid
 n_{i} \in \mathbb{N}\}$ forms a
Lyndon basis of the algebra $U^+$.
\end{theorem}

Recall that $\tau$ (see \cite{HRZ}) as a $\mathbb
Q$-anti-automorphism of $U_{r,s}(G_2)$ such that $\tau(r)=s$,
$\tau(s)=r$, $\tau(\langle \om_i',\om_j\rangle^{\pm1})=\langle
\om_j',\om_i\rangle^{\mp1}$, and
\begin{gather*}
\tau(e_i)=f_i, \quad \tau(f_i)=e_i, \quad \tau(\om_i)=\om_i',\quad
\tau(\om_i')=\om_i.
\end{gather*}
Using $\tau$ to $U^+$, we can get those negative quantum root
vectors in $U^-$. Define $F_{i}=\tau( E_{i})=f_i$ for $1\leq i\leq
2$, and
\begin{eqnarray}
F_{12}=\tau(E_{12})=f_2f_1-r^{3}f_1f_2,
\\
F_{112}=\tau(
E_{112})=F_{12}f_1-r^{2}sf_1F_{12}, \\
F_{1112}=\tau(
E_{1112})=F_{112}f_1-rs^{2}f_1F_{112},\\
F_{11212}=\tau(E_{11212})=F_{12}F_{112}-rs^{2}F_{112}F_{12}.
\end{eqnarray}
\begin{coro}\label{2.5}
$\{F_1^{m_1
}F_{1112}^{m_2}F_{112}^{m_3}F_{11212}^{m_4}F_{12}^{m_5}F_2^{m_6}\mid
m_{i} \in \mathbb{N}\}$ forms a Lyndon basis of the algebra $U^-$.
\end{coro}

\section{Restricted two-parameter quantum groups}
From now on, we restrict the parameters $r$ and $s$ to be roots of
unity: $r$ is a primitive $d$th root of unity, $s$ is a primitive
$d'$th root of unity and $\ell$ is the least common multiple of $d$
and $d'$. We suppose that $\mathbb{K}$ contains a primitive $\ell
$th root of unity. We will construct a finite-dimensional Hopf
algebra $\mathfrak{u}_{r,s}(G_2)$ of dimension $\ell^{16}$ as a
quotient of $U_{r,s}(G_2)$ by a Hopf ideal $\mathcal{I}$, which is
generated by certain central elements.

\subsection{Commutation relations in $U^{+}$ and central elements}
We give some commutation relations, which hold in the positive part
of two-parameter quantum group of type $G_2$, and the relations are
actually useful for determining central elements in  this subsection
and integrals in section 6.

In the following lemmas, we adopt the following notational
conventions: $\xi=r^2-s^2+rs$, $\eta=r^2-s^2-rs$ and $\ze
=(r^3-s^3)(r+s)^{-1}$, and we will need the $r,s$-integers,
factorials and binomial coefficients defined for positive integers
$i$,  $n$ and $m$ by
$$[n]_i:=\frac{r^{in}-s^{in}}{r^i-s^i},\quad [n]_1:=[n],$$  $$[n]!:=[n][n-1]\cdots
[2][1],\quad
 \left[n\atop m\right]:=\frac{[n]!}{[m]![n-m]!}. $$ By
convention $[0]=0$ and $[0]!=1$.

\begin{lemm}\label{3.1}\footnote {The proof of (6) is different from the one
given in Lemma 3.7 in \cite{HS}.} The following relations hold in
$U^{+}:$

$(1)\quad E_{112}e_2=(rs)^{3}e_2E_{112}+r(r^2{-}s^2)E_{12}^2;$

$(2)\quad
E_{11212}e_2=(r^2s)^3e_2E_{11212}+r^3(r{-}s)(r^2{-}s^2)E_{12}^3;$

$(3)\quad E_{1112}e_2=(rs^2)^3e_2E_{1112}+(r^2s)(r^3{-}s^3)
E_{12}E_{112} +r\eta E_{11212};$

$(4)\quad E_{1112}E_{12}=(rs)^3E_{12}E_{1112}+r\ze E_{112}^2;$

$(5)\quad e_1E_{11212}=(rs)^3E_{11212}e_{1}+r\ze E_{112}^2;$

$(6)\quad E_{1112}E_{112}=r^3E_{112}E_{1112};$

$(7)\quad E_{1112}E_{11212}=(r^2s)^3E_{11212}E_{1112}+r^3\ze
(r{-}s)E_{112}^3.$

$(8)\quad E_{11212}E_{12}=r^3E_{12}E_{11212};$

$(9)\quad E_{112}E_{11212}=r^3E_{11212}E_{112}.$
\end{lemm}
\begin{proof}
(1) follows directly from (2.3), (2.6) and (2.2).

\smallskip

(2) follows from (2.5), (2.6) and  (1).

\smallskip
(3): Using (2.3), we have $$
e_1E_{12}^2=E_{112}E_{12}+rs^2E_{12}E_{112}+(rs^2)^2E_{12}^2e_1.$$
Furthermore, using (2.4), (1) and (2.2), (2.5), we have
\begin{equation*}
\begin{split}
E_{1112}e_2&=
(e_1E_{112}-r^2sE_{112}e_1)e_2 \\
&=(rs)^3e_1e_2E_{112}+r(r^2{-}s^2)e_1E_{12}^2-r^2sE_{112}e_1e_2\\
&=(rs)^3(E_{12}+s^3e_2e_1)E_{112}+r(r^2{-}s^2)e_1E_{12}^2-r^2sE_{112}(E_{12}+s^3e_2e_1)\\
&=(rs)^3E_{12}E_{112}+(rs^2)^3e_2e_1E_{112}+r(r^2{-}s^2)e_1E_{12}^2-r^2sE_{112}E_{12}\\&\quad
-
(rs^2)^2((rs)^3e_2E_{112}+r(r^2{-}s^2)E_{12}^2)e_1\\
&=(rs)^3E_{12}E_{112}+(rs^2)^3e_2e_1E_{112}
+r(r^2{-}s^2)E_{112}E_{12}\\& \quad+
(rs)^2(r^2{-}s^2)E_{12}E_{112}+r^3s^4(r^2{-}s^2)E_{12}^2e_1
-r^2sE_{112}E_{12} \\&
\quad+(rs^2)^3e_2E_{1112}-(rs^2)^3e_2e_1E_{112}
-r^3s^4(r^2{-}s^2)E_{12}^2e_1\\
&=(rs^2)^3e_2E_{1112}+(r^2s)(r^3{-}s^3) E_{12}E_{112} +r\eta
E_{11212}.
\end{split}
\end{equation*}

\smallskip
(4): Using (2.2), (2.7) and (3), we have
\begin{equation*}
\begin{split}
E_{1112}E_{12}&=
E_{1112}(e_1e_2-s^3e_2e_1) \\
&=r^{-3}e_1E_{1112}e_2-s^3E_{1112}e_2e_1\\
&=r^{-3}e_1((rs^2)^3e_2E_{1112}+(rs)^2\xi E_{12}E_{112} +r\eta
E_{112}E_{12})\\& \quad -s^3((rs^2)^3e_2E_{1112}+(rs)^2\xi
E_{12}E_{112}
+r\eta E_{112}E_{12})e_1\\
&=s^6e_1 e_2E_{1112}+r^{-1}s^2\xi e_1E_{12}E_{112} +r^{-2}\eta
e_1E_{112}E_{12}\\
&\quad -(rs^3)^3e_2E_{1112}e_1-r^2s^5\xi E_{12}E_{112}e_1 -rs^3\eta
E_{112}E_{12}e_1\\& =s^6E_{12}E_{1112}+r^{-1}s^2\xi E_{112}^2+s^4\xi
E_{12}e_1E_{112}+r^{-2}\eta E_{1112}E_{12}\\
&\quad +s\eta E_{112}e_1E_{12}-r^2s^5\xi E_{12}E_{112}e_1-rs^3\eta
E_{112}E_{12}e_1\\& = s^6E_{12}E_{1112}+s^4\xi
E_{12}E_{1112}+r^{-1}s^2\xi E_{112}^2+r^{-2}\eta
E_{1112}E_{12}+s\eta E_{112}^2,
\end{split}
\end{equation*}
that is,
$$(r{+}s)E_{1112}E_{12}=(rs)^3(r{+}s)E_{12}E_{1112}+r(r^3{-}s^3)E_{112}^2.$$
Since we have assumed that $r^2\neq s^2$, this implies (4).

\smallskip
(5): Using (2.5), (2.4) and (2.3), we have
$$e_1E_{11212}-(rs)^3E_{11212}e_1=
E_{1112}E_{12}-(rs)^3E_{12}E_{1112},$$  and using (4), we have the
desired result.

\smallskip
(6): Using (2.3), (2.7) and (4), we have
\begin{equation*}
\begin{split}
E_{1112}E_{112}=s^3E_{112}E_{1112}+r^{-2}\ze
(e_1E_{112}^2-(r^2s)^2E_{112}^2e_1),
\end{split}
\end{equation*}
on the other hand, using (2.4), we have
\begin{equation*}
\begin{split}
E_{1112}E_{112}=
e_1E_{112}^2-r^2sE_{112}E_{1112}-(r^2s)^2E_{112}^2e_1,
\end{split}
\end{equation*}
so that owing to $r^2+s^2\ne 0$, we obtain
$$E_{1112}E_{112}=r^3E_{112}E_{1112}.$$

\smallskip
(7) can be proved by using (2.5), (6) and (4).

\smallskip
(8): It is easy to check the following relations:
$$e_1E_{12}^3-(rs^2)^3E_{12}^3e_1=E_{112}E_{12}^2+rs^2E_{12}E_{11212}
+r^2s^3(r{+}s)E_{12}^2E_{112},\eqno(*)$$
$$E_{112}^2e_2-(rs)^6e_2E_{112}^2=r^4s^3(r^2{-}s^2)E_{12}^2E_{112}
+r(r^2{-}s^2)E_{112}E_{12}^2,\eqno(**)$$
$$E_{112}E_{12}^2-(r^2s)^2E_{12}^2E_{112}=E_{11212}E_{12}+r^2sE_{12}E_{11212}.\eqno(***)$$
Using (2.2), (5) and (2), we have
\begin{equation*}
\begin{split}
E_{11212}E_{12}&=E_{11212}(e_1e_2-s^3e_2e_1)\\
&=(rs)^{-3}(e_1E_{11212}-r\ze E_{112}^2)e_2\\
& \quad
-s^3((r^2s)^3e_2E_{11212}+r^3(r{-}s)(r^2{-}s^2)E_{12}^3)e_1\\
&=s^{-3}e_1((rs)^3e_2E_{11212}+(r{-}s)(r^2{-}s^2)E_{12}^3)\\& \quad
-r^{-2}s^{-3}\ze E_{112}^2e_2-(rs)^3e_2(e_1E_{11212} -r\ze
E_{112}^2)\\& \quad -(rs)^3(r{-}s)(r^2{-}s^2)E_{12}^3e_1\\
&=r^3E_{12}E_{11212}\\&\quad
+s^{-3}(r{-}s)(r^2{-}s^2)(E_{112}E_{12}^2+rs^2E_{12}E_{11212}+r^2s^3(r{+}s)E_{12}^2E_{112})\\&
\quad -r^{-2}s^{-3}\ze
(r^4s^3(r^2{-}s^2)E_{12}^2E_{112}+r(r^2{-}s^2)E_{112}E_{12}^2)\\
&=r^3E_{12}E_{11212}-(rs)^{-1}(r{-}s)^2(E_{112}E_{12}^2-(r^2s)^2E_{12}^2E_{112})\\&\quad
+rs^{-1}(r{-}s)(r^2{-}s^2)E_{12}E_{11212}\\
&=r^3E_{12}E_{11212}-(rs)^{-1}(r{-}s)^2(E_{11212}E_{12}+r^2sE_{12}E_{11212})\\&\quad
+rs^{-1}(r{-}s)(r^2{-}s^2)E_{12}E_{11212},
\end{split}
\end{equation*}
so that,
$$(1{+}r^{-1}s^{-1}(r{-}s)^2)E_{11212}E_{12}=(r^3{+}rs^{-1}(r{-}s)(r^2{-}s^2)
-r(r{-}s)^2)E_{12}E_{11212}.$$
Since we have assumed that $r^2+s^2-rs\neq 0$, this implies
$$E_{11212}E_{12}=r^3E_{12}E_{11212}.$$

\medskip
(9): On the one hand, we have
\begin{equation*}
\begin{split}
E_{112}E_{11212}&=(e_1E_{12}-rs^2E_{12}e_1)E_{11212}\\&
=r^{-3}e_1E_{11212}E_{12}-rs^2E_{12}e_1E_{11212}\\&
=r^{-3}((rs)^3E_{11212}e_{1}+r\ze E_{112}^2)E_{12}\\&\quad
-rs^2E_{12}((rs)^3E_{11212}e_{1}+r\ze E_{112}^2)\\&=
s^3E_{11212}E_{112}+r^{-2}\ze
(E_{112}^2E_{12}-(r^2s)^2E_{12}E_{112}^2),
\end{split}
\end{equation*}
on the other hand, we have
\begin{equation*}
\begin{split}
E_{112}^2E_{12}=E_{112}E_{11212}+r^2sE_{11212}E_{112}+(r^2s)^2E_{12}E_{112}^2,
\end{split}
\end{equation*}
hence, we obtain $$E_{112}E_{11212}=r^3E_{11212}E_{112},$$ since
$r^2+s^2\neq 0.$
\end{proof}

The main result of this subsection is the following theorem.
\begin{theorem}\label{3.2}
$E_{\a}^{\ell}$, $F_{\a}^{\ell}$ $(\a \in \Phi^+)$
 and $\omega_k^{\ell}-1,\ (\omega_k')^{\ell}-1 \
( k=1, 2)$ are central in $U_{r,s}(G_2)$.
\end{theorem}
The proof of Theorem \ref{3.2} will be achieved through a sequence
of lemmas.

\begin{lemm}\label{3.3} Let $x, y, z$ be elements of a  $\mathbb{K}$-algebra such that  $yx=\a
xy+z$  for some $\a \in \mathbb{K}$ and $n$ a natural number. Then
the following assertions hold.

$(1)$ If  $zx=\b xz$ for some $\b(\neq \a)  \in \mathbb{K}$, then
$yx^n=\a^nx^ny+\frac{\a^n-\b^n}{\a-\b}x^{n-1}z$;

$(2)$ If  $yz=\b zy$ for some $\b(\neq \a)  \in \mathbb{K}$, then
$y^nx=\a^nxy^n+\frac{\a^n-\b^n}{\a-\b}zy^{n-1}$. \hfill\qed
\end{lemm}

\begin{lemm}\label{3.4} For any positive integer $a$, the following equalities
hold.

$(1)\quad E_{12}^a e_{2}=r^{3a}e_{2}E_{12}^a;$

$(2)\quad  e_1E_{1112}^a=r^{3a}E_{1112}^a e_1;$

$(3)\quad
e_1E_{112}^a=(r^2s)^{a}E_{112}^ae_1+r^{2(a-1)}[a]E_{112}^{a-1}E_{1112};$

$(4)\quad e_1E_{11212}^a=(rs)^{3a}E_{11212}^{a}e_1+r^{3a-2}\ze
[a]_3E_{11212}^{a-1}E_{112}^2;$

$(5)\quad
E_{11212}^ae_2=(r^2s)^{3a}e_2E_{11212}^{a}+r^{3(2a-1)}(r{-}s)(r^2{-}s^2)
[a]_3E_{12}^3E_{11212}^{a-1};$

$(6)\quad e_1 e_{2}^a=s^{3a}e_2^ae_1+[a]_3e_2^{a-1}E_{12}.$
\end{lemm}
\begin{proof}
(1) \& (2) follow from (2.6) and (2.7).
\smallskip

(3) follows from Lemma \ref{3.3} (1), where $x=E_{112}, y=e_1,
z=E_{1112}, \a =r^2s$, $\b=r^3$, and
$e_1E_{112}=E_{1112}+r^2sE_{112}e_1$,
$E_{1112}E_{112}=r^3E_{112}E_{1112}$.

\smallskip
(4) follows from Lemmas \ref{3.3} (1),  \ref{3.1} (5) \& (9), where
$x=E_{11212}, y=e_1, z=r\zeta E_{112}^2, \a =(rs)^3$, $\b=r^6$, and
$e_1E_{11212}=(rs)^3E_{11212}e_1+r\zeta E_{112}^2$ and
$E_{112}E_{11212}=r^3E_{11212}E_{112}$.

\smallskip
(5) follows from Lemmas \ref{3.1} (2) \& \ref{3.3} (2), where
$y=E_{11212}, x=e_2, z=r^3(r{-}s)(r^2{-}s^2) E_{12}^3, \a =(r^2s)^3$
and $\b=r^9$.

\smallskip
(6) follows from Lemma 3.3 (1), together with (2.2), (2.6).
\end{proof}

\begin{lemm}\label{3.5} For any positive integer $a$, the following equalities
hold.

\smallskip
$(1)\, E_{112}^ae_2=(rs)^{3a}e_2E_{112}^a+r^{3(a-2)}(r^2{-}s^2)
\Bigl(r^4s^{2(a-1)}[a]E_{12}^2E_{112}^2+r^3s^{a{-}2}[2]\left[a\atop
2\right]$

$\qquad\qquad\quad \times E_{12}E_{11212}
E_{112}+[2]\left[a\atop3\right]E_{11212}^2\Bigr)E_{112}^{a-3};$

\smallskip
$(2)\; e_1^ae_2=s^{3a}e_2e_1^a+s^{2(a-1)}[a]
E_{12}e_1^{a-1}{+}s^{a-2}\left[a\atop 2\right]E_{112}e_1^{a-2}{+}
{\left[a\atop 3\right]}E_{1112}e_1^{a-3},$

$\qquad\qquad\qquad\textit{for } \, a>4;$

\smallskip
$(3)\; e_1
E_{12}^a=(rs^2)^aE_{12}^ae_1+(rs)^{a-1}[a]E_{12}^{a-1}E_{112}
+r^{a-2}\left[a\atop 2\right]E_{12}^{a-2}E_{11212};$

\smallskip
$(4)\;
E_{1112}^ae_2=(rs^2)^{3a}e_2E_{1112}^a+(r^2s)(rs)^{3(a-1)}(r^3{-}s^3)[a]_3
E_{12}E_{112}E_{1112}^{a-1}$

$\qquad\quad\qquad\ +\,r\eta(rs)^{3(a-1)}[a]_3
E_{11212}E_{1112}^{a-1}+r^{3(a-1)}\ze(r{-}s)[2]_3\left[a\atop
2\right]_3E_{112}^3E_{1112}^{a-2}.$
\end{lemm}
\begin{proof} (1): Use induction on $a$. If $a=1$, this is
Lemma \ref{3.1} (1). Using
$E_{112}E_{12}^2=(r^2s)^2E_{12}^2E_{112}+r^2[2]E_{12}E_{11212}$,
(2.5) and Lemma \ref{3.1} (9), we can prove (1) to be true for
$a=2,3,4$. Suppose that (1) is true for all $a\geq 4$, we obtain
\begin{equation*}
\begin{split}
E_{112}^{a+1}e_2 &=(rs)^{3a}E_{112}e_2E_{112}^a+
r^{3(a-2)}(r^2{-}s^2)
\Bigl(r^4s^{2(a-1)}[a]E_{112}E_{12}^2E_{112}^2\\
&\quad +\, r^3s^{a{-}2}[2]\left[a\atop
2\right]E_{112}E_{12}E_{11212}
E_{112}+[2]\left[a\atop3\right]E_{112}E_{11212}^2\Bigr)E_{112}^{a-3}
\\
&=(rs)^{3(a+1)}e_2E_{112}^{a+1}+r(rs)^{3a}(r^2{-}s^2)E_{12}^2E_{112}^a\\
&\quad
+\,r^{3(a-2)}(r^2{-}s^2)\Bigl(r^8s^{2a}[a]E_{12}^2E_{112}^3+r^6s^{2a-2}
[2][a]E_{12}E_{11212}E_{112}^2\\
& \quad +\,r^3s^{a-2}[2]\left[a\atop 2\right]E_{11212}^2E_{112}
+r^8s^{a-1}[2]\left[a\atop 2\right]E_{12}E_{11212} E_{112}^2\\&\quad
+\,r^6[2]\left[a\atop 3\right]E_{11212}^2E_{112}\Bigr)E_{112}^{a-3}\\
&=(rs)^{3(a+1)}e_2E_{112}^{a+1}+ r^{3(a-1)}(r^2{-}s^2)\Bigl(
r^4s^{2a}[a{+}1]E_{12}^2E_{112}^2
\\
& \quad +\,r^3s^{a-1}[2]\left[a{+}1\atop 2\right]E_{12}E_{11212}
E_{112}+[2]\left[a{+}1\atop 3\right] E_{11212}^2\Bigr)E_{112}^{a-2}
\end{split}
\end{equation*}
by the induction hypothesis.

(2): We use induction on $a$. If $a=4$, using the defining relations
(2.2), (2.3), (2.4), \& (2.7),
and a simple computation, we have
$$e_1^4e_2=[4]E_{1112}e_1+s^2\left[4\atop 2\right]E_{112}e_1^2+s^6[4]E_{12}e_1^3+s^{12}e_2e_1^4.$$
Furthermore, using the induction hypothesis and (2.7), we obtain
\begin{equation*}
\begin{split}
e_1^{a+1}e_2 &= \left[a{+}1\atop
3\right]E_{1112}e_1^{a-2}+s^{a-1}\left[a{+}1\atop
2\right]E_{112}e_1^{a-1} \\
&\qquad+\,s^{2a}[a{+}1] E_{12}e_1^{a}+s^{3(a+1)}e_2e_1^{a+1}.
\end{split}
\end{equation*}

(3): We use induction on $a$. If $a=2$, using (2.3), we have
$$
e_1E_{12}^2=E_{11212}+rs[2]E_{12}E_{112}+(rs^2)^2E_{12}^2e_1.
$$
Suppose that (3) is true for $a\ge 2$, using the induction
hypothesis, (2.5) and Lemma \ref{3.1} (8), we obtain
\begin{equation*}
\begin{split}
e_1E_{12}^{a+1} &
=(rs^2)^aE_{12}^ae_1E_{12}+(rs)^{a-1}[a]E_{12}^{a-1}E_{112}E_{12}
+r^{a-2}\left[a\atop 2\right]E_{12}^{a-2}E_{11212}E_{12}\\
&= (rs^2)^aE_{12}^aE_{112}+(rs^2)^{a+1}E_{12}^{a+1}e_1
+(rs)^{a-1}[a]E_{12}^{a-1}E_{11212}\\
& \quad+r^{a+1}s^{a}[a]E_{12}^aE_{112}
+r^{a+1}\left[a\atop 2\right]E_{12}^{a-1}E_{11212}\\
&=(rs^2)^{a+1}E_{12}^{a+1}e_1
+(rs)^a[a{+}1]E_{12}^{a}E_{112}+r^{a-1}\left[a{+}1\atop
2\right]E_{12}^{a-1}E_{11212}.
\end{split}
\end{equation*}

(4): We use induction on $a$. If $a=1$, this is the relation in
Lemma \ref{3.1} (3).

Suppose that (4) is true  for all $a\ge 1$,
using the induction hypothesis and Lemma \ref{3.1} (3), (4), (6) \&
(7), we obtain
\begin{equation*}
\begin{split}
E_{1112}^{a+1}e_2&=
(rs^2)^{3a}\Bigl(E_{1112}e_2\Bigr)E_{1112}^a\\
&\quad+\,(r^2s)(rs)^{3(a-1)}(r^3{-}s^3)[a]_3
\Bigl(E_{1112}E_{12}\Bigr)E_{112}E_{1112}^{a-1}\\
&\quad+\,r\eta(rs)^{3(a-1)}[a]_3
\Bigl(E_{1112}E_{11212}\Bigr)E_{1112}^{a-1}\\
&\quad+\,r^{3(a-1)}\ze(r{-}s)[2]_3\left[a\atop
2\right]_3\Bigl(E_{1112}E_{112}^3\Bigr)E_{1112}^{a-2}\\
&=(rs^2)^{3a}\Bigl((rs^2)^3e_2E_{1112}+(r^2s)(r^3{-}s^3)E_{12}E_{112}+r\eta
E_{11212}\Bigr)E_{1112}^a\\
&\quad+\,(r^2s)(rs)^{3(a-1)}(r^3{-}s^3)[a]_3
\Bigl((rs)^3E_{12}E_{1112}+r\zeta E_{112}^2\Bigr)E_{112}E_{1112}^{a-1}\\
&\quad+\,r\eta(rs)^{3(a-1)}[a]_3
\Bigl((r^2s)^3E_{11212}E_{1112}+r^3\zeta(r{-}s)E_{112}^3\Bigr)E_{1112}^{a-1}\\
&\quad+\,r^{3a+6}\ze(r{-}s)[2]_3\left[a\atop
2\right]_3E_{112}^3E_{1112}^{a-1}\\
&
=(rs^2)^{3(a+1)}e_2E_{1112}^{a+1}+(r^2s)(rs)^{3a}(r^3{-}s^3)[a{+}1]_3
E_{12}E_{112}E_{1112}^a\\
&\quad+\,r\eta(rs)^{3a}[a{+}1]_3
E_{11212}E_{1112}^a+r^{3a}\ze(r{-}s)[2]_3\left[a{+}1\atop
2\right]_3E_{112}^3E_{1112}^{a-1}.
\end{split}
\end{equation*}

This completes the proof.
\end{proof}

\begin{lemm}\label{3.6}
\ For any positive integer $a$, the following equalities hold.

$(1)\quad e_k^af_k=f_ke_k^a+[a]_k\ e_k^{a-1}\frac{s_k^{-a+1}\omega_k
-r_k^{-a+1}\omega_k'}{r_k-s_k},\quad 1\leq k \leq 2;$

$(2)\quad E_{12}^af_{1}=f_{1}E_{12}^a
-r^{2(a-1)}[3][a]e_2E_{12}^{a-1}\om_1';$

$(3)\quad E_{112}^af_{1}=f_{1}E_{112}^a
-(rs)^{a-1}[2]^2[a]E_{12}E_{112}^{a-1}\omega_{1}'
-r^{a-2}[2]^2\left[a\atop2\right]E_{11212}E_{112}^{a-2}\omega_{1}';$

$(4)\quad E_{1112}^af_{1}=f_{1}E_{1112}^a-[3]{[a]_3}
E_{112}E_{1112}^{a-1}\omega_{1}';$

$(5)\quad
E_{11212}^af_{1}=f_{1}E_{11212}^a-r^{3a-2}(r^2{-}s^2)[3]{[a]_3}
E_{12}^2E_{11212}^{a-1}\omega_{1}'.$
\end{lemm}
\begin{proof}
(1) can be easily verified by induction on $a$.

(2), (4) \& (5) follow from Lemma \ref{3.3} (2), and the following
\begin{gather*}
E_{12}f_1=f_1E_{12}-[3]e_2\om_1',\\
E_{112}f_1=f_1E_{112}-[2]^2E_{12}\om_1',\\
E_{1112}f_1=f_1E_{1112}-[3]E_{112}\om_1',\\
E_{11212}f_1=f_1E_{11212}-r(r^2{-}s^2)[3]E_{12}^2\om_1'.
\end{gather*}

(3) can be verified by induction on $a$.
\begin{equation*}
\begin{split}
E_{112}^{a+1}f_1&=\Bigl(E_{112}f_{1}\Bigr)E_{112}^a
-(rs)^{a-1}[2]^2[a]\Bigl(E_{112}E_{12}\Bigr)E_{112}^{a-1}\omega_{1}'\\
&\quad-\,r^{a-2}[2]^2\left[a\atop2\right]\Bigl(E_{112}E_{11212}\Bigr)E_{112}^{a-2}\omega_{1}'\\
&=f_{1}E_{112}^{a+1} -(rs)^a[2]^2[a{+}1]E_{12}E_{112}^a\omega_{1}'
-r^{a-1}[2]^2\left[a{+}1\atop2\right]E_{11212}E_{112}^{a-1}\omega_{1}'.
\end{split}
\end{equation*}

This completes the proof.
\end{proof}

\begin{lemm}\label{3.7}
\ For any positive integer $a$, the following equalities hold.

$(1)\; E_{12}^af_{2}=f_{2}E_{12}^a+\om_2E_{12}^{a-3}\Bigl
(s^{2(a-1)}E_{12}^{2}e_1+s^{a-2}\left[a\atop 2\right]
E_{12}E_{112}+\left[a\atop3\right]E_{11212}\Bigr);$

$(2)\;
E_{112}^af_{2}{=}f_{2}E_{112}^a+r^{3(a-2)}(r^2{-}s^2)\om_2E_{112}^{a-3}
\Bigl([2]\left[a\atop3\right]E_{1112}^2
+r^{3}s^{a-2}[2]\left[a\atop2\right]E_{112}\times$

$\qquad\qquad\qquad \times
E_{1112}e_1+[a]r^{4}s^{2(a-1)}E_{112}^{2}e_1^2\Bigr);$

$(3)\;
E_{1112}^af_{2}=f_{2}E_{1112}^a+r^{-3(a-2)}(r^2{-}s^2)(r{-}s)[a]_3\om_2e_1^3E_{1112}^{a-1};$

$(4)\;
E_{11212}^af_{2}=f_{2}E_{11212}^a+r(rs)^{3(a-1)}[a]_3\om_2E_{11212}^{a-1}\Bigr(\eta
E_{1112}+rs(r^3{-}s^3)E_{112}e_1\Bigl)$

$\qquad\qquad\qquad +\,r^{3(a-1)}\ze
(r{-}s)[2]_3\left[a\atop2\right]_3\om_2E_{11212}^{a-2}E_{112}^3.$
\end{lemm}
\begin{proof}
(1): Since $E_{12}f_2=f_2E_{12}+\om_2e_1$, it is easy to check the
cases when $a=2,3,4$. Using the induction hypothesis, and by Lemma
\ref{3.5} (3), we have
\begin{gather*}
\begin{split}
E_{12}^{a{+}1}f_2&=E_{12}f_{2}E_{12}^a+\om_2E_{12}^{a-2}\Bigl(s^{2a+1}[a]E_{12}^{2}e_1
{+}s^{a+1}\left[a\atop2\right]
E_{12}E_{112}{+}s^3\left[a\atop3\right]E_{11212}\Bigr)
\\
&=f_{2}E_{12}^{a+1}+r^{a-2}\om_2\Bigl(r^2s^{2a}E_{12}^ae_1{+}rs^{a-1}[a]E_{12}^{a-1}E_{112}
{+}\left[a\atop2\right]E_{12}^{a-2}E_{11212}\Bigr)\\& \quad
+\om_2E_{12}^{a-2}\Bigl(s^{2a+1}[a]E_{12}^{2}e_1
+s^{a+1}\left[a\atop2\right]
E_{12}E_{112}+s^3\left[a\atop3\right]E_{11212}\Bigr)\\
&=f_{2}E_{12}^{a{+}1}+\om_2E_{12}^{a-2}\Bigl(s^{2a}[a{+}1]E_{12}^{2}e_1{+}s^{a{-}1}\left[a{+}1\atop2\right]
E_{12}E_{112}\\
&\quad+\left[a{+}1\atop3\right]E_{11212}\Bigr).
\end{split}
\end{gather*}

(2): From (1), it is easy to calculate
$E_{112}f_2=f_2E_{112}+r(r^2{-}s^2)\om_2e_1^2$. The relation below
is obtained by Lemma \ref{3.4} (3),
$$
e_1^2E_{112}^a=r^{4a-6}E_{112}^{a-2}\left([2]\left[a\atop2\right]E_{1112}^2
+r^{4}s^{a-1}[2][a]E_{112}E_{1112}e_1+r^{6}s^{2a}E_{112}^2e_1^2\right).$$

For $a\ge 1$, we have by induction on $a$,
\begin{gather*}
\begin{split}
E_{112}^{a+1}f_{2}&=\Bigl(E_{112}f_{2}\Bigr)E_{112}^a+r^{3(a-1)}s^3(r^2{-}s^2)\om_2E_{112}^{a-2}
\Bigl([2]\left[a\atop3\right]E_{1112}^2\\& \quad
+r^{3}s^{a-2}[2]\left[a\atop2\right]
E_{112}E_{1112}e_1+r^{4}s^{2(a-1)}[a]E_{112}^{2}e_1^2\Bigr)\\
&=f_{2}E_{112}^{a+1}+r(r^2{-}s^2)r^{4a-6}\om_2E_{112}^{a-2}\Bigl(
[2]\left[a\atop2\right]E_{1112}^2
\\& \quad +r^{4}s^{a-1}[2][a]E_{112}E_{1112}e_1+r^{6}s^{2a}E_{112}^2e_1^2\Bigr)\\&
\quad +r^{3(a-1)}s^3(r^2{-}s^2)\om_2E_{112}^{a-2}
\Bigl([2]\left[a\atop3\right]E_{1112}^2\\
& \quad
+r^{3}s^{a-2}[2]\left[a\atop2\right]
E_{112}E_{1112}e_1+r^{4}s^{2(a-1)}[a]E_{112}^{2}e_1^2\Bigr)\\
&=f_{2}E_{112}^{a+1}+r^{3(a-1)}(r^2{-}s^2)\om_2E_{112}^{a-2}
\Bigl([2]\left[a{+}1\atop3\right]E_{1112}^2\\& \quad
+r^{3}s^{a-1}[2]\left[a{+}1\atop2\right]
E_{112}E_{1112}e_1+r^{4}s^{2a}[a{+}1]E_{112}^{2}e_1^2\Bigr).
\end{split}
\end{gather*}

(3): We consider the case for $a=1$. It follows from (2) that
$$
E_{1112}f_2=f_2E_{1112}+r^3(r^2{-}s^2)(r{-}s)\om_2e_1^3.\eqno (*)
$$
So by  $(*)$ and Lemma \ref{3.3} (2), we can easily get (3).

(4): If $a=1$, we have $$
E_{11212}f_2=f_2E_{11212}+r\eta\om_2E_{1112}+r^2s(r^3{-}s^3)\om_2E_{112}e_1$$
by using (1) and (2). The following relation holds by Lemma 3.3 (1):
$$E_{1112}E_{11212}^a=(r^2s)^{3a}E_{11212}^aE_{1112}
+r^{3(2a-1)}\ze (r{-}s)[a]_3E_{11212}^{a-1}E_{112}^3. \eqno (**)$$
So (4) can be derived easily  from $(**)$ and Lemma \ref{3.4} (4) by
induction on $a$.

The proof is complete.
\end{proof}

\textbf{Proof of Theorem \ref{3.2}.} We know from Lemmas \ref{3.4},
\ref{3.5}, \ref{3.6}, \ref{3.7} that the elements $E_{1}^{\ell}$,
$E_{1112}^{\ell}$, $E_{112}^{\ell}$, $E_{11212}^{\ell}$, $
E_{12}^{\ell}$, $ E_{2}^{\ell}$ are central in $U_{r,s}(G_2)$.
Applying $\tau$ to $E_{\a}\ (\a \in \Phi)$, we see that
$F_{1}^{\ell}$, $F_{1112}^{\ell}$, $F_{112}^{\ell}$,
$F_{11212}^{\ell}$, $ F_{12}^{\ell}$, $ F_{2}^{\ell}$ are also
central. It is easy to see that $\omega_k^{\ell}-1,\
(\omega_k')^{\ell}-1 \ (k=1, 2)$ are central too. $\hfill\qed$

\begin{remark}\label{3.8}
If $\ell=3\ell'$, then the elements $E_{1112}^{\ell'},
F_{1112}^{\ell'}$, $\omega_k^{\ell'}{-}1$ and
$(\omega_k')^{\ell'}{-}1$ $(k=1, 2)$ are central in $U_{r,s}(G_2)$.
\end{remark}

\subsection{Restricted two-parameter quantum groups}

From now on we will assume that $\ell$ is coprime to $3$.

\begin{defi}\label{3.9} The \textit{restricted two-parameter quantum group} is the quotient
$$\mathfrak{u}_{r,s}(G_2):=U_{r,s}(G_2)/\mathcal{I},$$
where   $\mathcal{I}$ denotes the ideal of $U_{r,s}(G_2)$ generated
by all elements $E_{\a}^{\ell}$, $F_{\a}^{\ell}$ $(\a \in \Phi^+)$
and $\omega_k^{\ell}-1,\ (\omega_k')^{\ell}-1 \ (1\leq k\leq 2 )$.
\end{defi}

By Theorem \ref{2.4} and Corollary \ref{2.5},
$\mathfrak{u}_{r,s}(G_2)$ is an algebra of dimension $\ell^{16}$
with linear basis
$$
E_{2}^{c_{1}}E_{12}^{c_{2}}E_{11212}^{c_3}E_{112}^{c_4}E_{1112}^{c_5}E_{1}^{c_6}
\omega_1^{b_1}\omega_{2}^{b_{2}}
(\omega'_1)^{b'_1}(\omega'_{2})^{b'_{2}}F_{1}^{c'_{1}}F_{1112}^{c'_{2}}F_{112}^{c'_3}F_{11212}^{c'_4}F_{12}^{c'_5}
F_{2}^{c'_6}\leqno(3.1)
$$
where all powers range between $0$ and $\ell-1$.

The remainder of this section is devoted to proving that
$\mathfrak{u}_{r,s}(G_2)$ is a finite-dimensional Hopf algebra.

First note that the generators of $\mathcal{I}$ are contained in the
kernel of the counit $\varepsilon$, and so $\mathcal{I}$ is as well.
Because the coproduct $\Delta$ is an algebra homomorphism, and the
antipode $S$ is an algebra antihomomorphism, it suffices to show
that $\Delta(x)\in\mathcal{I}\ot U+U\ot\mathcal{I}$ and $S(x)\in
\mathcal{I}$ for each generator $x\in\mathcal{I}$. To accomplish
this, we rely on the following computations:
\begin{gather*}
\begin{split}
\Delta(\omega_k^{\ell}-1)&=\om_k^\ell\ot\om_k^\ell-1\ot
1\\
&=(\omega_k^{\ell}-1)\ot\om_k^\ell+1\ot(\omega_k^{\ell}-1)\in\mathcal{I}\ot
U+U\ot\mathcal{I},
\end{split}
\end{gather*}
and
$S(\omega_k^{\ell}-1)=-\om_k^{-\ell}(\omega_k^{\ell}-1)\in\mathcal{I}.$
The argument for $(\omega_k')^{\ell}-1$ is similar. In determining
$\Delta(x)$ for each generator $x \in \mathcal{I}$, we adopt the
notation below
$$\om_{12}=\om_1\om_2, \quad\om_{112}=\om_1^2\om_2,\quad \om_{1112}=\om_1^3\om_2, \quad \om_{11212}=\om_1^3\om_2^2.$$

Some simple computations give rise to the following
\begin{lemm}\label{3.10}
$(1)\ \Delta (E_{12})=E_{12}\otimes 1 + \omega_{12}\otimes E_{12}+
(r^3{-}s^3)\omega_{2}e_1 \otimes e_2;$

$(2)\ \Delta (E_{112})=E_{112}\otimes 1 + \omega_{112}\otimes
E_{112}+ (r{+}s)(r^2{-}s^2)\omega_{12}e_1 \otimes
E_{12}+r(r^2{-}s^2)\times (r^3{-}s^3)\omega_{2}e_1^2 \otimes e_{2};$

$(3)\ \Delta (E_{1112})=E_{1112}\otimes 1 + \omega_{1112}\otimes
E_{1112}+(r^3{-}s^3)\omega_{112}e_1\otimes E_{112}+
r(r^2{-}s^2)\times (r^3{-}s^3)\omega_{12}e_1^2 \otimes
E_{12}+r^3(r{-}s)(r^2{-}s^2)(r^3{-}s^3)\omega_{2}e_1^3 \otimes
e_{2};$

$(4)\ \Delta (E_{11212})=E_{11212}\otimes 1 + \omega_{11212}\otimes
E_{11212} +r(r^2{-}s^2)(r^3{-}s^3)\omega_{12}^2e_1 \otimes E_{12}^2
+r(r^2{-}s^2)(r^3{-}s^3)^2\omega_{12}\om_{2}e_1^2 \otimes
E_{12}e_{2} +(r^3{-}s^3)\omega_{12}E_{112} \otimes E_{12}
+r^6(r{-}s)(r^2{-}s^2)\times (r^3{-}s^3)^2\omega_{2}^2e_1^3 \otimes
e_{2}^2
+r(r^3{-}s^3)\om_2((r^2{-}s^2{-}rs)E_{1112}+rs(r^3{-}s^3)E_{112}e_1)\ot
e_2.$\hfill\qed
\end{lemm}

\begin{lemm}\label{3.11}
 $\Delta(E_{\a}^\ell)\in\mathcal{I}\ot
U+U\ot\mathcal{I}$, for $\a \in \Phi^+$.
\end{lemm}
\begin{proof}
Note that
$$
\Delta(e_i^n)=\sum\limits_{j=0}^{n}\left(n\atop
j\right)_{r_is_i^{-1}} e_i^{n-j}\omega_i^{j}\ot e_i^{j}
$$
has a $r_is_i^{-1}$-binomial expansion
\footnote{ We set
$$
(n):=(n)_{rs^{-1}}:=\frac{r^{n}s^{-n}-1}{rs^{-1}-1},  \quad (n)!:=(n)(n-1)\cdots
(2)(1),\quad
 \left(n\atop m\right):=\frac{(n)!}{(m)!(n-m)!}.$$
By convention $(0)=0$ and $(0)!=1$.} since $(\omega_i\ot
e_i)(e_i\ot1)=r_is_i^{-1}(e_i\ot 1)(\omega_i\ot e_i)$. So we have
$\Delta(e_i^\ell)=e_i^\ell\ot 1+\om_i^\ell\ot e_i^\ell \in
\mathcal{I}\ot U+U\ot\mathcal{I}$, since $r$ and $s$ are primitive
$\ell$th root of unity, and $\left(\ell\atop
j\right)_{r_is_i^{-1}}=0$ for $0< j< \ell$.

Note that $\Delta(E_{12}^\ell)=(E_{12}\ot 1+(r^3-s^3)\omega_2e_1\ot
e_2)^\ell+\omega_{12}^\ell\ot E_{12}^\ell,$ because of
$$
(\omega_{12}\ot E_{12})(E_{12}\ot 1+(r^3{-}s^3)\omega_2e_1\ot
e_2)=rs^{-1}(E_{12}\ot 1+(r^3{-}s^3)\omega_2e_1\ot
e_2)(\omega_{12}\ot E_{12}).
$$
we get $\Delta(E_{12}^\ell)=E_{12}^\ell\ot
1+r^{\frac{3\ell(\ell-1)}{2}}(r^3{-}s^3)^\ell\om_2^\ell e_1^\ell \ot
e_2^\ell+\om_{12}^\ell\ot E_{12}^\ell \in \mathcal{I}\ot
U+U\ot\mathcal{I}$, as $r$, $s$ are primitive $\ell$th root of
unity, and $(E_{12}+\om_2
e_1)^\ell=E_{12}^\ell+r^{\frac{3\ell(\ell-1)}{2}}\om_2^\ell
e_1^\ell$.  In a similar way, we have $\Delta(E_{112}^\ell),
\Delta(E_{1112}^\ell), \Delta(E_{11212}^\ell)$ are in
$\mathcal{I}\ot U+U\ot\mathcal{I}$.
\end{proof}

Applying the antipode property to $\Delta(E_{\a}^{\ell}),\  \a \in
\Phi^+$, we can obtain $S(E_{\a}^\ell)\in \mathcal{I}\ot
U+U\ot\mathcal{I}$. Using $\tau$, we find that $\Delta(F_{\a}^\ell)
\in \mathcal{I}\otimes U+U\otimes \mathcal{I}$ for $\a \in \Phi^+$.
Hence, $\mathcal{I}$ is a Hopf ideal, and $\mathfrak{u}_{r,s}(G_2)$
is a finite-dimensional Hopf algebra. Then we have

\begin{theorem}\label{3.12}
$\mathcal{I}$ is a Hopf ideal of $U_{r,s}(G_2)$, so that
$\mathfrak{u}_{r,s}(G_2)$ is a finite-dimensional Hopf algebra.
\hfill\qed
\end{theorem}

\section{Isomorphisms of $\mathfrak{u}_{r,s}(G_2)$}
Write $\mathfrak{u}=\mathfrak{u}_{r,s}=\mathfrak{u}_{r,s}(G_2)$. Let
$G$ denote the group generated by $\omega_i$, $\omega'_i\ ( i=1, 2)$
in the restricted quantum group $\mathfrak{u}$. Define linear
subspace $\mathfrak{a}_k$ of $\mathfrak{u}$ by
\begin{gather*}
\begin{split}
\mathfrak{a}_0=\mathbb{K}G, \qquad \mathfrak{a}_1=\mathbb{K}G
+\sum\limits_{i=1}^{2}(\mathbb{K}e_iG+\mathbb{K}f_iG), \\
\mathfrak{a}_k=(\mathfrak{a}_1)^k \quad \textit{for}\quad  k\geq 1.
\end{split}\tag{4.1}
\end{gather*}
Note that $1 \in \mathfrak{a}_0$, $\Delta(\mathfrak{a}_0)\subseteq
\mathfrak{a}_0\otimes \mathfrak{a}_0$, $\mathfrak{a}_1$ generates
$\mathfrak{u}$ as an algebra, and $\Delta(\mathfrak{a}_1)\subseteq
\mathfrak{a}_1\otimes \mathfrak{a}_0+\mathfrak{a}_0\otimes
\mathfrak{a}_1$. By \cite{M}, $\{\mathfrak{a}_k\}$ is a coalgebra
filtration of $\mathfrak{u}$ and
$\mathfrak{u}_0\subseteq\mathfrak{a}_0$, where the coradical
$\mathfrak{u}_0$ of $\mathfrak{u}$ is the sum of all the simple
subcoalgebras of $\mathfrak{u}$. Clearly, $\mathfrak{a}_0\subseteq
\mathfrak{u}_0$ as well, and so $\mathfrak{u}_0=\mathbb{K}G$. This
implies that $ \mathfrak{u}$ is \textit{pointed}, that is, every
simple subcoalgebra of $\mathfrak{u}$ is one-dimensional.

Let $\mathfrak{b}$ be the Hopf subalgebra of
$\mathfrak{u}=\mathfrak{u}_{r,s}(G_2)$ generated by $e_i,
\omega_i^{\pm 1}\ ( i=1, 2)$, and $\mathfrak{b}'$ the Hopf
subalgebra generated by $f_i, (\omega_i')^{\pm 1}\ (i=1, 2)$. The
same argument shows that $\mathfrak{b}$ and $\mathfrak{b}'$ are
pointed as well. Thus, we have
\begin{prop}\label{4.1}
The restricted two-parameter quantum group $\mathfrak{u}_{r,s}(G_2)$
is a pointed Hopf algebra, as are its subalgebras $\mathfrak{b}$ and
$\mathfrak{b}'$.
\end{prop}

Lemma 5.5.1 in \cite{M} indicates that $\mathfrak{a}_k\subseteq
\mathfrak{u}_k$ for all $k$, where $\{\mathfrak{u}_k\}$ is the
\textit{coradical filtration} of $\mathfrak{u}$ defined inductively
by $\mathfrak{u}_k=\Delta^{-1}(\mathfrak{u}\times
\mathfrak{u}_{k-1}+\mathfrak{u}_0\times \mathfrak{u}).$ In
particular, $\mathfrak{a}_1\subseteq \mathfrak{u}_1.$ By Theorem
5.4.1 in \cite{M}, as $\mathfrak{u}$ is pointed, $\mathfrak{u}_1$ is
spanned by the set of group-like elements $G$ together with all the
skew-primitive elements of $\mathfrak{u}$. We claim that under the
additional hypothesis of the Lemma below,
$\mathfrak{u}_1=\mathfrak{a}_1$. That is, each skew-primitive
element of $u$ is a linear combination of elements of $G$, of
$e_iG$, and of $f_iG\ (i=1, 2)$.

The following Lemma gives a precise description of $\mathfrak{u}_1$.
\begin{lemm}\label{4.2}
Assume that $rs^{-1}$ is a primitive $\ell$th root of unity. Then
$$\mathfrak u_1=\mathbb{K}G
+\sum\limits_{i=1}^{2}(\mathbb{K}e_iG+\mathbb{K}f_iG).$$
\end{lemm}

Given two group-like elements $g, h$ in a Hopf algebra $H$, let
$P_{g,h}(H)$ denote the set of skew-primitive elements of $H$ given
by
$$P_{g,h}(H)=\{x\in H\mid \Delta(x)=x\otimes g+h\otimes x\}.$$ We
want to compute $P_{1,\sigma}(\mathfrak{u}_{r,s})$ and
$P_{\sigma,1}(\mathfrak{u}_{r,s})$ for $\sigma \in G$.

\begin{lemm}\label{4} Assume that $rs^{-1}$ is a primitive $\ell$th root of
unity. Then we have

$(1)$\quad
$P_{1,\omega_i}(\mathfrak{u}_{r,s})=\mathbb{K}(1-\omega_i)+\mathbb{K}e_i$,
\quad $
P_{1,\omega_i'^{-1}}(\mathfrak{u}_{r,s})=\mathbb{K}(1-\omega_i'^{-1})+\mathbb{K}f_i\omega_i'^{-1}$,

$\qquad \ P_{1,\sigma}(\mathfrak{u}_{r,s})=\mathbb{K}(1-\sigma)
\quad \textrm{if}\quad
 \sigma\not\in \{\omega_i, \omega_i'^{-1}\mid i=1, 2\,\};$

$(2)\quad
P_{\omega'_i,1}(\mathfrak{u}_{r,s})=\mathbb{K}(1-\omega'_i)+\mathbb{K}f_i$,
\quad $
P_{\omega_i^{-1},1}(\mathfrak{u}_{r,s})=\mathbb{K}(1-\omega_i^{-1})+\mathbb{K}e_i\omega_i^{-1}$,

$\qquad \ P_{\sigma,1}(\mathfrak{u}_{r,s})=\mathbb{K}(1-\sigma)
\quad \textrm{if}\quad
 \sigma\not\in\{ \omega'_i, \omega_i^{-1} \mid i=1, 2\,\}.$
\end{lemm}
\begin{proof} (1)
Assuming $rs^{-1}$ is a primitive $\ell$th  root of unity, we have
$$\mathbb{K}G+\sum\limits_{g,h \in G}P_{g,h}(\mathfrak{u}_{r,s})=\mathbb{K}G
+\sum\limits_{i=1}^{2}(\mathbb{K}e_iG+\mathbb{K}f_iG).$$ Thus, any
$x\in P_{1,\sigma}(\mathfrak{u}_{r,s})$ (where $\sigma \in G$) may
be written as a linear combination
$$ x=\sum\limits_{g\in G}\gamma_gg+\sum\limits_{g\in
G}\sum\limits_{i=1}^{2}\alpha_{i,g}e_ig+\sum\limits_{g\in
G}\sum\limits_{i=1}^{2}\beta_{i,g}f_ig,$$ where $\gamma_g,
\alpha_{i,g}$ and $\beta_{i,g}$ are scalars. Comparing
$\Delta(x)=x\otimes 1+\sigma\otimes x$ with the coproduct of the
right side, which is
$$\sum\limits_{g\in
G}\gamma_gg\otimes g+\sum\limits_{g\in
G}\sum\limits_{i=1}^{2}\alpha_{i,g}(e_ig\otimes g+\omega_ig\otimes
e_ig)+\sum\limits_{g\in G}\sum\limits_{i=1}^{2}\beta_{i,g}(g\otimes
f_ig+f_ig\otimes \omega'_ig),$$ we obtain $\beta_{i,g}=0$ for all
$i$ and $g\ne \omega_i'^{-1}$, and $\alpha_{i,g}=0$ for all $g\neq
1$. A further comparison of the group-like components yields
$\gamma_{\sigma}=-\gamma_1$  and $\gamma_g=0$ for all $g\neq
1,\sigma$. Finally, comparing $\Delta(x)$ with
$\Delta(\gamma_1(1-\sigma)+\sum\limits_{i=1}^2\alpha_{i,1}e_i+\sum\limits_{i=1}^2\beta_{i,\omega_i'^{-1}}f_i\omega_i'^{-1})$
yields  $\alpha_{i,1}=0$ for all $i$ when $\sigma \notin
\{\omega_1,\omega_2\}$, and $\beta_{i,\omega_i'^{-1}}=0$ for all $i$
when $\sigma \notin \{\omega_1'^{-1},\omega_2'^{-1}\}$.

(i) When $\sigma=\omega_i$ for some $i$, we have $\alpha_{j,1}=0$
for all $j\neq i$; and $\beta_{j,\omega_j'^{-1}}=0$ for all $j$. So
$x=\gamma_1(1-\omega_i)+\alpha_{i,1}e_i$.

(ii) When $\sigma=\omega_i'^{-1}$ for some $i$, we have
$\beta_{j,\omega_j'^{-1}}=0$ for all $j\neq i$; and $\alpha_{j,1}=0$
for all $j$. So
$x=\gamma_1(1-\omega_i'^{-1})+\beta_{i,\omega_i'^{-1}}f_i\omega_i'^{-1}$.
Thus
\begin{gather*}
\begin{split}
P_{1,\omega_i}(\mathfrak{u}_{r,s})=\mathbb{K}(1-\omega_i)+\mathbb{K}e_i,\quad 1\leq i\leq 2,\\
P_{1,\omega_i'^{-1}}(\mathfrak{u}_{r,s})=\mathbb{K}(1-\omega_i'^{-1})+\mathbb{K}f_i\omega_i'^{-1},\quad
1\leq i\leq 2,\\
P_{1,\sigma}(\mathfrak{u}_{r,s})=\mathbb{K}(1-\sigma), \quad
\textit{if } \
 \sigma\not\in \{\omega_i, \omega_i'^{-1}\}, \quad \textit{for any} \  i.\\
\end{split}
\end{gather*}

Similarly, we get (2).
\end{proof}

\begin{remark}
It should be mentioned that the description of the set of left
(right) skew-primitive elements of $\mathfrak
u_{r,s}(\mathfrak{sl}_n)$ (see formulae (3.6) \& (3.7) in
\cite{BW3}) is not complete. This led to many families of
isomorphisms of $\mathfrak u_{r,s}(\mathfrak{sl}_n)$ undiscovered,
while in the second part of the proof of Theorem 5.4 in \cite{HW2},
authors described explicitly all of the Hopf algebra isomorphisms of
$\mathfrak u_{r,s}(\mathfrak{sl}_n)$ (see Theorem 5.5 in
\cite{HW2}).
\end{remark}

\begin{theorem}\label{4.3}
Assume that $rs^{-1}$ and $r's'^{-1}$ are primitive $\ell$th roots
of unity with $\ell\ne 3, 4$, and $\zeta$ is a $3$rd root of unity.
Then $\varphi: \mathfrak{u}_{r,s}\cong\mathfrak{u}_{r',s'}$ as Hopf
algebras if and only if either $(\text{\rm i})$ $(r', s')=\zeta(r,
s)$, $\varphi$ is a diagonal isomorphism$:$
$\varphi(\omega_i)=\tilde \omega_i$,
$\varphi(\omega_i')=\tilde\omega_i'$, $\varphi(e_i)=a_i\tilde e_i$,
$\varphi(f_i)=\zeta^{\delta_{i,1}} a_i^{-1}\tilde f_i$
$(a_i\in\mathbb K^*)$; or $(\text{\rm ii})$ $(r', s')=\zeta(s, r)$,
$\varphi(\omega_i)=\tilde\omega_i'^{-1}$,
$\varphi(\omega_i')=\tilde\omega_i^{-1}$, $\varphi(e_i)=a_i \tilde
f_i\tilde\omega_i'^{-1}$, $\varphi(f_i)=\zeta^{\delta_{i,1}}
a_i^{-1}\tilde\omega_i^{-1}\tilde e_i$ $(a_i\in\mathbb K^*)$.
\end{theorem}
\begin{proof}
Suppose $\varphi: \mathfrak{u}_{r,s}\longrightarrow
\mathfrak{u}_{r',s'}$ is a Hopf algebra isomorphism. Write
$\tilde{e}_i,\, \tilde{f}_i,\, \tilde{\omega}_i$,
$\tilde{\omega'_i}$ to distinguish the generators of
$\mathfrak{u}_{r',s'}$. Because
$$
\Delta(\varphi(e_i))=(\varphi\otimes\varphi)
(\Delta(e_i))=\varphi(e_i)\otimes 1+\varphi(\omega_i)\otimes
\varphi(e_i),
$$
we have $\varphi(e_i) \in
P_{1,\varphi(\omega_i)}(\mathfrak{u}_{r',s'})$ for $i=1,2$. As
$\varphi$ is an isomorphism,
$\varphi(\omega_i)\in\mathbb{K}\tilde{G}$. Lemma \ref{4} (1) implies
that $\varphi(\omega_i)\in\{\,\tilde\omega_j,\,
\tilde\omega_j'^{-1}\mid j=1,2\,\}$, that is, case (i): For $i=1,2$,
we have $\varphi(\omega_i)=\tilde{\omega}_{j_i}$ for some $j_i$, and
$\varphi(e_i)=\alpha(1-\tilde{\omega}_{j_i})+\beta \tilde{e}_{j_i}$
($\alpha, \beta \in \mathbb{K}$); \ case (ii): For $i=1,2$, we have
$\varphi(\omega_i)=\tilde\omega_{j_i}'^{-1}$ for some $j_i$, and
$\varphi(e_i)=\alpha'(1-\tilde{\omega}_{j_i}'^{-1})+\beta'
\tilde{f}_{j_i}\tilde{\omega}_{j_i}'^{-1}$ ($\alpha', \beta' \in
\mathbb{K}$); \ case (iii): $\varphi(\omega_1)=\tilde \omega_{j_1}$,
but $\varphi(\omega_2)=\tilde\omega_{j_2}'^{-1}$; \ case (iv):
$\varphi(\omega_1)=\tilde \omega_{j_1}'^{-1}$, but
$\varphi(\omega_2)=\tilde\omega_{j_2}$.

\smallskip
Case (i): In this case, we claim
$\varphi(\omega_i)=\tilde{\omega}_i$. Otherwise, if
$\varphi(\omega_1)=\tilde{\omega}_{2}$ and
$\varphi(\omega_2)=\tilde{\omega}_1$. Applying $\varphi$ to relation
$\omega_1e_1=rs^{-1}e_1\omega_1$ yields
\begin{gather*}
\varphi(\omega_1)\varphi(e_1)=rs^{-1}\varphi(e_1)\varphi(\omega_1),\\
\tilde{\omega}_{2}(\alpha(1-\tilde{\omega}_{2})+\beta
\tilde{e}_{2})=rs^{-1}(\alpha(1-\tilde{\omega}_{2})+\beta
\tilde{e}_{2})\tilde{\omega}_{2}.
\end{gather*}
As $r\neq s$, this forces $\alpha$ to be $0$, so that
$\varphi(e_1)=\beta \tilde{e}_{2}$ for some $\beta\neq 0$. Moreover,
because
$\tilde{\omega}_{2}\tilde{e}_{2}=r_2's_2'^{-1}\tilde{e}_{2}\tilde{\omega}_{2}$,
it follows that $r'^3s'^{-3}=rs^{-1}$.

Applying $\varphi$ to relation $\omega_1e_2=s^3e_2\omega_1$, we get
$\varphi(\omega_1)\varphi(e_2)=s^3\varphi(e_2)\varphi(\omega_1)$,
i.e., $\tilde{\omega}_{2}\tilde{e}_{1}=s^3
\tilde{e}_{1}\tilde{\omega}_{1}$. But
$\tilde{\omega}_{2}\tilde{e}_{1}=r'^{-3}
\tilde{e}_{1}\tilde{\omega}_{1}$ in $\mathfrak{u}_{r',s'}$. So
$r'^3=s^{-3}$. Similarly, to $\omega_2e_1=r^{-3}e_1\omega_2$ to get
$s'^{-3}=r^3$. Thereby, it follows from $r'^3s'^{-3}=rs^{-1}$ that
$r^2=s^2$, which contradicts the assumption. So we proved
$\varphi(\omega_i)=\tilde{\omega}_i$. Similarly, we have
$\varphi(\omega_i')=\tilde\omega_i'$.

Furthermore, applying $\varphi$ to relations $\omega_ie_j=\langle
\omega_j',\omega_i\rangle e_j\omega_i$ for $i, j\in\{1, 2\}$, we get
$\langle \omega_j',\omega_i\rangle=\langle
\tilde\omega_j',\tilde\omega_i\rangle$ for $i, j\in\{1, 2\}$, i.e.,
$r'^3s'^{-3}=r^3s^{-3}$, $r's'^{-1}=rs^{-1}$, $r'^{-3}=r^{-3}$ and
$s'^3=s^3$. Hence, $(r', s')=\zeta(r, s)$, where $\zeta$ is a $3$rd
root of unity. As $\varphi$ preserves relation $(G4)$,
$\varphi(e_i)=a_i\tilde e_i$ and $\varphi(f_i)=\zeta^{\delta_{i,1}}
a_i^{-1}\tilde f_i$ ($a_i\in\mathbb K^*$), that is, $\varphi$ is a
diagonal isomorphism.

\smallskip
Case (ii): Since $\varphi(\omega_i)=\tilde\omega_{j_i}'^{-1}$ for
some $j_i$, and
$\varphi(e_i)=\alpha'(1-\tilde{\omega}_{j_i}'^{-1})+\beta'
\tilde{f}_{j_i}\tilde{\omega}_{j_i}'^{-1}$. In fact, we have
$\varphi(e_i)=\beta' \tilde f_{j_i}\tilde\omega_{j_i}'^{-1}$
(applying $\varphi$ to relation $\omega_ie_i=r_is_i^{-1}e_i\omega_i$
to get $\alpha'=0$). We claim
$\varphi(\omega_i)=\tilde\omega_i'^{-1}$. In fact, if
$\varphi(\omega_1)=\tilde\omega_2'^{-1}$ and
$\varphi(\omega_2)=\tilde\omega_1'^{-1}$, then $\varphi(e_1)=a_1
\tilde f_2\tilde\omega_2'^{-1}$ and $\varphi(e_2)=a_2 \tilde
f_1\tilde\omega_1'^{-1}$. Using $\varphi$ to relation
$\omega_1e_1=rs^{-1}e_1\omega_1$, we get
$rs^{-1}=(r'^3s'^{-3})^{-1}$; to relation
$\omega_2e_1=r^{-3}e_1\omega_2$ to get $r^{-3}=r'^3$, to relation
$\omega_1e_2=s^3e_2\omega_1$ to get $s^3=s'^{-3}$. So we have
$(rs^{-1})^2=1$. This contradicts the assumption. Hence,
$\varphi(\omega_i)=\tilde\omega_i'^{-1}$,
$\varphi(\omega_i')=\tilde\omega_i^{-1}$, $\varphi(e_i)=a_i \tilde
f_i\tilde\omega_i'^{-1}$, where $a_i\in\mathbb K^*$.

Using $\varphi$ to relation $\omega_1e_1=rs^{-1}e_1\omega_1$, we get
$rs^{-1}=(r's'^{-1})^{-1}$; to relation $\omega_1e_2=s^3e_2\omega_1$
to get $s^3=r'^3$; to relation $\omega_2e_1=r^{-3}e_1\omega_2$ to
get $r^{-3}=s'^{-3}$. So we have $(r', s')=\zeta(s, r)$, where
$\zeta$ is a $3$rd root of unity. In order to $\varphi$ preserve
relation $(G4)$, one has to take $\varphi(f_i)=\zeta^{\delta_{i,1}}
a_i^{-1}\tilde\omega_i^{-1}\tilde e_i$, where $a_i\in\mathbb K^*$.
Besides, we can easily check that such a $\varphi$ does preserve all
of the Serre relations $(G5)$---$(G6)$.

\smallskip
Case (iii): We claim that this case is impossible. First, we assume
that $\varphi(\omega_1)=\tilde \omega_1$ and
$\varphi(\omega_2)=\tilde\omega_2'^{-1}$. Then we have
$\varphi(e_1)=a_1\tilde e_1$, $\varphi(e_2)=a_2\tilde
f_2\tilde\omega_2'^{-1}$, where $a_1, a_2\in\mathbb K^*$. Using
$\varphi$ to relation $\omega_1e_2=s^3e_2\omega_1$, we get
$s^3=s'^{-3}$; to relation  $\omega_2e_1=r^{-3}e_1\omega_2$ to get
$r^{-3}=s'^{3}$. So we get $r^3=s^3$. This is a contradiction.

Next, we assume that $\varphi(\omega_1)=\tilde \omega_2$, and
$\varphi(\omega_2)=\tilde\omega_1'^{-1}$. Then we have
$\varphi(e_1)=a_1\tilde e_2$, $\varphi(e_2)=a_2\tilde
f_1\tilde\omega_1'^{-1}$, where $a_1, a_2\in\mathbb K^*$. Using
$\varphi$ to relation $\omega_1e_2=s^3e_2\omega_1$, we
 get $s^3=r'^3$; to relation
$\omega_2e_1=r^{-3}e_1\omega_2$ to get $r^{-3}=r'^{-3}$. So we have
$r^3=s^3$, which is contrary to the condition.

\smallskip
Case (iv): Similarly to case (iii), it is also impossible.

So we complete the proof.
\end{proof}

\section{$\mathfrak{u}_{r,s}(G_2)$ is a Drinfel'd double}

We will show that $\mathfrak{u}\cong D(\mathfrak{b})$ under a few
assumptions. Let $\theta$ be a primitive $\ell$th root of unity in
$\mathbb{K}$, and write $r=\theta^y$, $s=\theta^z$.

\begin{lemm}\label{5.1}
Assume that $(3(y^2+z^2+yz),\ell)=1$ and $rs^{-1}$ is a primitive
$\ell$th root of unity. Then $(\mathfrak{b}')^{coop}\cong
\mathfrak{b}^*$ as Hopf algebras.
\end{lemm}
\begin{proof}
Define $\gamma_j, \eta_j \ (j=1, 2)$ of $\mathfrak{b}^*$ as follows:
The $\gamma_j$'s are algebra homomorphisms with
$\gamma_{j}(\omega_i)=\langle \omega_j',\omega_i\rangle$ and
$\gamma_{j}(e_i)=0$ for $i=1, 2$. So they are  group-like elements
in $\mathfrak{b}^*$. Let $\eta_j=\sum\limits_{g \in G(\mathfrak
b)}(e_jg)^*$, where $G(\mathfrak{b})$ is  the group generated by
$\omega_1, \omega_2$ and the asterisk denotes the dual basis element
relative to the PBW-basis of $\mathfrak{b}$. The isomorphism $\phi:
(\mathfrak{b}')^{coop}\longrightarrow \mathfrak{b}^*$ is then
defined by
$$\phi(\omega'_j)=\gamma_j \quad \textrm{and}\quad
\phi(f_j)=\eta_j.$$  We have to check  that $\phi$ is a Hopf algebra
homomorphism and  it is a bijection.

First,  we observe that the $\gamma_j$'s are invertible elements in
$\mathfrak{b}^*$ that commute with one another and
$\gamma_j^\ell=1$. Note that $\eta_j^\ell=0$, as it is $0$ on any
basis element of $\mathfrak{b}$. We calculate
$\gamma_j\eta_i\gamma_j^{-1}$: It is nonzero only on basis elements
of the form $e_i\omega_1^{k_1}\omega_2^{k_2}$, and on such an
element it takes the value
\begin{gather*}
\begin{split}
& (\gamma_j\otimes\eta_i\otimes \gamma_j^{-1})((e_i\otimes 1\otimes
1+ \omega_i\otimes e_i\otimes 1+\omega_i\otimes \omega_i\otimes
e_i)(\omega_1^{k_1}\omega_2^{k_2})^{\otimes 3}) \\&=
\gamma_j(\omega_i\omega_1^{k_1}\omega_2^{k_2})\eta_i(e_i\omega_1^{k_1}\omega_2^{k_2})
\gamma_j^{-1}(\omega_1^{k_1}\omega_2^{k_2})\\
&=\gamma_{j}(\omega_i)=\langle\omega_j',\omega_i\rangle.
\end{split}
\end{gather*}
Thus we have
$\gamma_j\eta_i\gamma_j^{-1}=\langle\omega_j',\omega_i\rangle\eta_i,$
which corresponds to relation $(G3)$ for $\mathfrak{b}'$. Next, we
check relation $(G6)$: we compute
\begin{gather*}
\begin{split}
& (\eta_2^2\eta_{1})(e_2^2e_{1})\\
 &=(\eta_2\otimes\eta_2\otimes\eta_{1})((e_2\otimes 1\otimes
1+\omega_2\otimes e_2\otimes 1+\omega_2\otimes \omega_2\otimes
e_2)^2\\  & \quad \times(e_{1}\otimes 1\otimes 1+\omega_{1}\otimes
e_{1}\otimes 1+\omega_{1}\otimes
\omega_{1}\otimes e_{1}))\\
\end{split}
\end{gather*}
\begin{gather*}
\begin{split}
&=(\eta_2\otimes\eta_2\otimes\eta_{1})(e_2\omega_2\omega_{1}\otimes
e_2\omega_{1}\otimes e_{1}+\omega_2e_2\omega_{1}\otimes
e_2\omega_{1}\otimes e_{1})\\
&=(\eta_2\otimes\eta_2\otimes\eta_{1})((1+r^3s^{-3})e_2\omega_2\omega_{1}\otimes
e_2\omega_{1}\otimes e_{1})=1+r^3s^{-3},
\end{split}
\end{gather*}
and similarly, $(\eta_2^2\eta_{1})(e_2E_{12})=0$. Thus, for any
$k_1,k_2$, we have $(\eta_2^2\eta_{1})(e_2^2$
$e_{1}\omega_1^{k_1}\omega_2^{k_2})$ $=1+r^3s^{-3}$ and
$(\eta_2^2\eta_{1})(e_2E_{12}\omega_1^{k_1}\omega_2^{k_2})=0$. On
all other basis elements, $\eta_2^2\eta_{1}$ is $0$. Therefore, we
have
$$\eta_2^2\eta_{1}=\sum\limits_{g \in
G}(1+r^3s^{-3})(e_2^2e_{1}g)^*. \eqno(*)$$ Similarly, we calculate
\begin{gather*}
\begin{split}
& (\eta_2\eta_{1}\eta_2)(e_2^2e_{1})\\
& \quad
=(\eta_2\otimes\eta_{1}\otimes\eta_{2})(e_2\omega_2\omega_{1}\otimes
\omega_{2}e_{1}\otimes e_{2}+\omega_2e_2\omega_{1}\otimes
\omega_{2}e_{1}\otimes e_{2})\\
&\quad =r^{-3}+s^{-3},\\
& (\eta_2\eta_{1}\eta_2)(e_2E_{12})\\
& \quad =(\eta_2\eta_{1}\eta_2)(e_2e_{1}e_2-s^3e_2^2e_{1})\\
 &
\quad
=(\eta_2\otimes\eta_{1}\otimes\eta_{2})(\omega_2\omega_{1}e_2\otimes
\omega_{2}e_{1}\otimes e_{2}+e_2\omega_1\omega_{2}\otimes
e_{1}\omega_{2}\otimes e_{2})-1-r^{-3}s^3\\
&\quad =1-r^{-3}s^3.
\end{split}
\end{gather*}
So we have $$\eta_2\eta_{1}\eta_2=\sum\limits_{g \in
G}((r^{-3}+s^{-3})(e_2^2e_{1}g)^*+(1-r^{-3}s^3)(e_2E_{12}g)^*).\eqno(**)$$
We compute
\begin{gather*}
\begin{split}
(\eta_{1}\eta_2^2)(e_2^2e_{1})&=r^{-3}(r^{-3}+s^{-3}),\\
(\eta_{1}\eta_2^2)(e_2E_{12})&=
s^3(s^{-6}-r^{-6}).
\end{split}
\end{gather*}
So we have $$\eta_{1}\eta_2^2=\sum\limits_{g \in
G}(r^{-3}(r^{-3}+s^{-3})(e_2^2e_{1}g)^*+s^3(s^{-6}-r^{-6})(e_2E_{12}g)^*).\eqno(***)$$
We use the above results in $(*)$---$(***)$ to establish the
relation
$$r^3s^3\eta_1\eta_{2}^2-(r^3+s^3)\eta_2\eta_{1}\eta_2+\eta_{2}^2\eta_1=0.$$
Similarly, it is easy to verify that
\begin{gather*}
\begin{split}
\eta_{2}\eta_1^{4} - (r +& s)(r^2 + s^2)\,\eta_1\eta_{2}\eta_1^{3} +
rs(r^2 + s^2)(r^2 + rs + s^2)\,\eta_{1}^{2}\eta_{2}\eta_1^2\\&
-(rs)^3(r + s)(r^{2} + s^2)\,\eta_{1}^3\eta_{2}\eta_1 +\,
(rs)^6\,\eta_{1}^4\eta_2 = 0.
\end{split}
\end{gather*}
Therefore, $\phi $ is an algebra homomorphism.

Now we will check that $\phi$ preserves coproducts. We have already
seen that $\gamma_i$ is a group-like element in $\mathfrak{b}^*$. We
calculate
\begin{gather*}
\begin{split}
\Delta(\eta_i)(e_i\omega_1^{j_1}\omega_2^{j_2}\otimes \omega_1^{k_1}
\omega_2^{k_2})&=\eta_i(e_i\omega_1^{j_1+k_1}\omega_2^{j_2+k_2})=1,\\
\Delta(\eta_i)(\omega_1^{j_1}\omega_2^{j_2}\otimes
e_i\omega_1^{k_1}\omega_2^{k_2})
&=\eta_i(\omega_1^{j_1}\omega_2^{j_2}e_i\omega_1^{k_1}
\omega_2^{k_2})
=\langle
\omega_i',\omega_1\rangle^{j_1}\langle\omega_i',\omega_2\rangle^{j_2}.
\end{split}
\end{gather*}
These are the only basis elements of $\mathfrak{b}\otimes
\mathfrak{b}$ on which $\Delta(\eta_i)$ is nonzero. As a
consequence, we have
\begin{gather*}
(\eta_i\otimes 1+\gamma_i\otimes
\eta_i)(e_i\omega_1^{j_1}\omega_2^{j_2}\otimes \omega_1^{k_1}
\omega_2^{k_2})=1,\\
(\eta_i\otimes 1+\gamma_i\otimes
\eta_i)(\omega_1^{j_1}\omega_2^{j_2}\otimes e_i\omega_1^{k_1}
\omega_2^{k_2})=\langle
\omega_i',\omega_1\rangle^{j_1}\langle\omega_i',\omega_2\rangle^{j_2}.
\end{gather*}
So $\Delta(\eta_i)=\eta_i\otimes 1+\gamma_i\otimes \eta_i$. This
proves that $\phi$ is a Hopf algebra homomorphism.

Finally, we prove that $\phi$ is bijective. As $\mathfrak{b}^*$ and
$(\mathfrak{b}')^{coop}$ have the same dimension, it suffices to
show that $\phi$ is injective. By \cite{M}, we need only show that
$\phi|_{(\mathfrak{b}')_1^{coop}}$ is injective. Lemma \ref{4.2}
yields
$(\mathfrak{b}')_1^{coop}=\mathbb{K}G(\mathfrak{b}')+\sum_{i=1}^{2}\mathbb{K}f_iG(\mathfrak{b}')$,
where $G(\mathfrak{b}')$ is the group generated by $\omega'_1,
\omega'_2$. First, we claim that
$$
\textrm{span}_\mathbb{K}\{
\gamma_1^{k_1}\gamma_2^{k_2}\mid 0\le
k_i<\ell\,\}=\textrm{span}_\mathbb{K}\{(\omega_1^{k_1}\omega_2^{k_2})^*\mid
0\le k_i<\ell\,\}. \quad \leqno (5.1)
$$
This is equivalent to the statement that the
$\gamma_1^{k_1}\gamma_2^{k_2}$'s span the space of characters over
$\mathbb{K}$ of the finite group $\mathbb{Z}/\ell \mathbb{Z}\times
\mathbb{Z}/\ell \mathbb{Z}$ generated by $\omega_1,\omega_2$. We
have assumed that $\mathbb{K}$ contains a primitive $\ell$th root of
unity. Therefore, the irreducible characters of this group are the
functions $\chi_{i_1,i_2}$ given by
$\chi_{i_1,i_2}(\omega_1^{k_1}\omega_2^{k_2})=\theta
^{i_1k_1+i_2k_2},$ where $\theta$ is a primitive $\ell$th root of
unity in $\mathbb{K}$. Note that $\gamma_1=\chi_{y-z,-3y},
\gamma_2=\chi_{3z,3(y-z)}.$  We must show that, given $i_1,i_2$,
there are $k_1,k_2$ such that
$\chi_{i_1,i_2}=\gamma_1^{k_1}\gamma_2^{k_2},$ which is equivalent
to the existence of a solution to the matrix equation $AK=I$ in
$\mathbb{Z}/\ell \mathbb{Z}$ (as these are powers of $\theta$),
where
$A= \begin{pmatrix} y-z & 3z \\
-3y &3(y-z)
\end{pmatrix},$
 $K$ is  the transpose of $(k_1,k_2)$ and $I$ is  the transpose of
 $(i_1,i_2)$. The determinant of  the coefficient matrix $A$
 is $3(y^2+z^2+yz)$, which is invertible in $\mathbb{Z}/\ell
 \mathbb{Z}$ by the hypothesis in the Lemma. Therefore (5.1) holds.
 In particular, this implies that the matrix $$\big(
 (\gamma_1^{k_1}\gamma_2^{k_2})(\omega_1^{j_1}\omega_2^{j_2})\big)_{\bar
 k\times \bar j} \leqno (5.2)$$ is invertible, and that $\phi$ is
 bijection on group-like elements.

 Next we will show for each $i\ (i=1, 2)$ that the following
 matrix is invertible:
$$\big(
 (\eta_i\gamma_1^{k_1}\gamma_2^{k_2})(e_i\omega_1^{j_1}\omega_2^{j_2})\big)_{\bar
 k\times \bar j} \leqno (5.3)$$ This will complete the proof that
 $\phi$ is injective on $(\mathfrak{b}')_1^{coop}$, as desired. We
 will show that the matrix is block upper-triangular. Each matrix
 entry is
\begin{gather*}
\begin{split}
&(\eta_i\otimes\gamma_1^{k_1}\gamma_2^{k_2})(\Delta(e_i)\Delta(\omega_1^{j_1}\omega_2^{j_2}))\\
&\quad=(\eta_i\otimes\gamma_1^{k_1}\gamma_2^{k_2})(e_i\omega_1^{j_1}\omega_2^{j_2}
\otimes\omega_1^{j_1}\omega_2^{j_2}).
\end{split}
\end{gather*}
Thus, (5.3) is precisely the invertible matrix (5.2).
\end{proof}

Recall that the \textit{Drinfel'd double} $D(\mathfrak{b})$ of the
finite-dimensional Hopf algebra $\mathfrak{b}$ is a Hopf algebra
whose underlying coalgebra is $\mathfrak{b}\otimes
(\mathfrak{b}^*)^{coop}$ (that is, the vector space
$\mathfrak{b}\otimes (\mathfrak{b}^*)^{coop}$ with the tensor
product coalgebra structure). As an algebra, $D(\mathfrak{b})$
contains the subalgebras $\mathfrak{b}\otimes 1\cong \mathfrak{b}$
and $1\otimes \mathfrak{b}^* \cong \mathfrak{b}^*$, and if $a \in
\mathfrak{b}$ and $b \in (\mathfrak{b}^*)^{coop}$, then $(a\otimes
1)(1\otimes b)=a \otimes b$ and
$$(1\otimes b)(a\otimes
1)=\sum b_{(1)}(S^{-1}a_{(1)})b_{(3)}(a_{(3)})a_{(2)}\otimes
b_{(2)},$$ where $S^{-1}$ is the composition inverse of the antipode
$S$ for $\mathfrak{b}$.

\begin{theorem}\label{5.2}
Assume that $(3(y^2+z^2+yz),\ell)=1$. Then
 $D(\mathfrak{b})\cong \mathfrak{u}_{r,s}(G_2)$ as Hopf algebras.
\end{theorem}
\begin{proof}
We denote the image $e_i\otimes 1$ of $e_i$ in $D(\mathfrak{b})$ by
$\check{e}_i$, and similarly for $\omega_i$, $\eta_i$ and
$\gamma_i$. Define $\psi:D(\mathfrak{b})\rightarrow
\mathfrak{u}_{r,s}(G_2)$ on  the generators by
\begin{gather*}
\begin{split}
\psi(\check{e}_i)&=e_i, \quad \psi(\check{\eta}_i)=(s_i-r_i)f_i,\\
\psi(\check{\omega}_i^{\pm 1})&=\omega_i^{\pm 1}, \quad
\psi(\check{\gamma}_i^{\pm 1})=(\omega'_i)^{\pm 1}.
\end{split}
\end{gather*}


Then the proof is the same as that of Theorem 6.2 in \cite{HW2}.
\end{proof}

\section{Integrals and ribbon elements}

Integrals play a basic role in the structure theory of finite
dimensional Hopf algebras $H$ and their duals $H^*$. In this
section, we compute the left and right integrals in the Borel
subalgebra $\mathfrak{b}$ of $\mathfrak{u}_{r,s}(G_2)$ and the
distinguished group-like elements of $\mathfrak{b}$ and
$\mathfrak{b}^*$. Furthermore, we use this information to determine
that $\mathfrak{u}_{r,s}(G_2)$ has  a ribbon element when
$\mathfrak{u}_{r,s}(G_2)\cong D(\mathfrak{b})$.

Let $H$ be a finite-dimensional Hopf algebra. An element $y \in H$
is a \textit{left integral} (resp. \textit{right integral}) if
$ay=\varepsilon(a)y$ (resp. $ya=\varepsilon(a)y$)  for all $a \in
H$. The left (resp. right) integrals form a one-dimensional ideal
$\int^l_H$ (resp. $\int^r_H$) of $H$, and $S_H(\int^r_H)=\int^l_H$
under the antipode $S_H$ of $H$.

When $y\neq 0$ is a left integral of $H$, there exists a unique
group-like element $\gamma$ in the dual algebra $H^*$ (the so-called
\textit{distinguished group-like element} of $H^*$) such that
$ya=\gamma(a)y$. If we had begun instead with a right integral
$y'\in H$, then we would have $ay'=\gamma^{-1}(a)y'$. This is an
easy consequence of the fact that group-like elements are
invertible, and  can be found in [\cite{M}, p.22].

Now if $\lambda \neq 0$ is right integral of $H^*$, then there
exists a unique group-like element $g$ of $H$ (the
\textit{distinguished group-like element} of $H$) such that
$\xi\lambda=\xi(g)\lambda$ for all $\xi \in H^*$. The algebra $H$ is
\textit{unimodular} (i.e., $\int^l_H=\int^r_H$) if and only if
$\gamma=\varepsilon$; and the dual algebra $H^*$ is unimodular if
and only if $g=1$.

The left and right $H^*$-module actions on $H$ are given by
$$
\xi\rightharpoonup a=\sum a_{(1)}\xi(a_{(2)}), \qquad
a\leftharpoonup\xi=\sum \xi(a_{(1)})a_{(2)}
$$
for all $\xi \in H^*$ and $a\in H$. In particular,
$\varepsilon\rightharpoonup a=a=a\leftharpoonup\varepsilon$ for all
$a \in H$. Radford \cite{R} found a remarkable expression relating
the antipode, the distinguished group-like elements $\gamma$ and
$g$, and the $H^*$-action:
$$S^4(a)=g(\gamma \rightharpoonup a\leftharpoonup
\gamma^{-1})g^{-1}, \qquad \textit{for}\,\, \textit{all} \  a \in
H.$$ This formula is crucial in \cite{KR}, where Kauffman and
Radford determine a necessary and sufficient condition for a
Drinfel'd double of a Hopf algebra to have a ribbon element.

A finite-dimensional Hopf algebra $H$ is \textit{quasitriangular} if
there exists an invertible element $R=\sum x_i\ot y_i$ in $H\ot H$
such that $\Delta^{op}(a)R=R\Delta(a)$ for all $a \in H$, and $R$
satisfies the relations $(\Delta\ot id)R=R_{1,3}R_{2,3}$,
$(id\ot\Delta )R=R_{1,3}R_{1,2}$, where $R_{1,2}=\sum x_i\ot y_i\ot
1,R_{1,3}=\sum x_i\ot 1\ot y_i$, and $R_{2,3}=\sum 1\ot x_i\ot y_i$.
Suppose $u=\sum S(y_i)x_i$. Then $c=uS(u)$ is central in $H$ and is
referred  to as the \textit{Casimir element}.

An element $v\in H$ is a \textit{quasi-ribbon element} of
quasitriangular Hopf algebra $(H,R)$ if
$$
\textrm{(i)}\ v^2=c, \quad \textrm{(ii)}\ S(v)=v, \quad
\textrm{(iii)}\ \vn(v)=1, \quad \textrm{(iv)}\
\Delta(v)=(R_{2,1}R_{1,2})^{-1}(v\ot v),
$$
where $R_{2,1}=\sum y_i\ot x_i$ and $R_{1,2}=R$. If moreover $v$ is
central in $H$, then $v$ is a \textit{ribbon element}, and $(H,R,v)$
is said to be a \textit{ribbon Hopf algebra}. Ribbon elements
provide an effective means of constructing invariants of knots and
links (see \cite{KR, RT1, RT2,RTF}). The Drinfel'd double $D(A)$ of
a finite-dimensional Hopf algebra $A$ is quasitriangular, and
Kauffman and Radford have provided a simple criterion for $D(A)$ to
have a ribbon element in \cite{KR}.

\begin{theorem}\label{6.1}
Assume that $A$ is a finite-dimensional Hopf algebra, and let $g$
and $\gamma$ be the distinguished group-like elements of $A$ and
$A^*$ respectively. Then

$(\mathrm{i})$ \ $(D(A),R)$ has a quasi-ribbon element if and only
if there exist group-like elements $h\in A$, $\delta \in A^*$ such
that $h^2=g$ and $\delta^2=\gamma$.

$(\mathrm{ii})$ \ $(D(A),R)$ has a ribbon element if and only if
there exist  $h$  and $\delta$  as in $(\mathrm{i})$ such that
$$S^2(a)=h(\delta\rightharpoonup a\leftharpoonup
\delta^{-1})h^{-1}, \quad\textit{ for all} \ a\in A.$$
\end{theorem}

Next we compute a left integral and a right integral in
$\mathfrak{b}$.

\begin{prop}\label{6.2} Let $t=\prod\limits_{i=1}^{2}(1+\omega_i+\cdots+\omega_i^{\ell-1})$
and $
x=E_{2}^{\ell-1}E_{12}^{\ell-1}E_{11212}^{\ell-1}E_{112}^{\ell-1}$
$E_{1112}^{\ell-1}E_{1}^{\ell-1}.$ Then  $y=tx$  and $y'=xt$ are
respectively  a left integral and a right integral in
$\mathfrak{b}$.
\end{prop}
\begin{proof}
To prove $y=tx$ is a left integral in $\mathfrak{b}$, we need to
show that $by=\varepsilon(b)y$ for all $b \in \mathfrak{b}$. It
suffices to show this for the generators $\omega_k$ and $e_k$ for
$k=1,2$, as the counit $\varepsilon$ is an algebra homomorphism.

Observe that $\omega_kt=t=\vn(\om_k)t$ for all $k=1,2$, as the
$\om_i$'s commute and
$\om_k(1+\om_k+\cdots+\om_k^{\ell-1})=1+\om_k+\cdots+\om_k^{\ell-1}$.
From that, the relation $\om_k y=\vn(\om_k)y$ is clear for all $k$.

Next we compute $e_ky$. By a simple calculation, we get
$e_kt=\prod\limits_{i=1}^{2}(1+\langle
\omega_k',\omega_i\rangle^{-1}\om_i$ $+ \cdots+\langle
\omega_k',\omega_i\rangle^{-(\ell-1)}\om_i^{\ell-1})e_k.$ So it
suffices to check that $e_kx=0=\vn(e_k)x$.

Note that
 $e_2x=0$, we want to show that $e_1$ can be moved across  the
terms $E_{2}^{\ell-1}E_{12}^{\ell-1}E_{11212}^{\ell-1}$
$E_{112}^{\ell-1}E_{1112}^{\ell-1}$ next to $ E_{1}^{\ell-1}$. We
have
\begin{gather*}
\begin{split}
e_1&E_{2}^{\ell-1}E_{12}^{\ell-1}
E_{11212}^{\ell-1}E_{112}^{\ell-1}E_{1112}^{\ell-1}\\&=
s^{3(\ell-1)}E_{2}^{\ell-1}e_1E_{12}^{\ell-1}E_{11212}^{\ell-1}E_{112}^{\ell-1}E_{1112}^{\ell-1}\\&=
r^{\ell-1}s^{5(\ell-1)}E_{2}^{\ell-1}E_{12}^{\ell-1}e_1E_{11212}^{\ell-1}E_{112}^{\ell-1}E_{1112}^{\ell-1}\\&=
r^{4(\ell-1)}s^{8(\ell-1)}E_{2}^{\ell-1}E_{12}^{\ell-1}E_{11212}^{\ell-1}e_1E_{112}^{\ell-1}E_{1112}^{\ell-1}\\&=
r^{6(\ell-1)}s^{9(\ell-1)}E_{2}^{\ell-1}E_{12}^{\ell-1}E_{11212}^{\ell-1}E_{112}^{\ell-1}e_1E_{1112}^{\ell-1}\\&=
r^{9(\ell-1)}s^{9(\ell-1)}E_{2}^{\ell-1}E_{12}^{\ell-1}E_{11212}^{\ell-1}E_{112}^{\ell-1}E_{1112}^{\ell-1}e_1,
\end{split}
\end{gather*}
where we get the first "=" by using Lemma \ref{3.4} (6), the second
by using Lemmas \ref{3.5} (3) \& \ref{3.1} (9), the third by using
Lemma \ref{3.4} (4), the forth by using Lemma \ref{3.4} (3) and the
last by using Lemma \ref{3.4} (2). Thus, we have $e_1 x=0$, which
implies the desired conclusion that $y$ is a left integral in
$\mathfrak{b}$.

To prove that $y'=xt$ is a right integral in $\mathfrak{b}$,  it
suffices to show that $xe_j=0$.  Note that $xe_1=0$. By a similar
argument and  Lemma \ref{3.5} (2), (4), (1), Lemma \ref{3.4} (5),
(1) and Lemma \ref{3.1} (4), (6), (8) and (9), together with (2.5)
\& (2.6), we can move $e_2$ to the left until it is next to
$E_{2}^{\ell-1}$, which gives zero.
\end{proof}

A finite-dimensional Hopf algebra $H$ is semisimple if and only if
$\vn(\int_{H}^l)\neq 0$ if and only if $\vn(\int_{H}^r)\neq 0$. For
the algebra $\mathfrak{b}$ above, $y$ gives a basis for
$\int_{\mathfrak{b}}^l$ and $y'$ a basis for
$\int_{\mathfrak{b}}^r$. As $\vn(y)=0=\vn(y')$, we have

\begin{prop}\label{6.3}
The Hopf algebra $\mathfrak{b}$ is not semisimple. \hfill\qed
\end{prop}

We will compute the distinguished group-like elements of
$\mathfrak{b}$ and $\mathfrak{b}^*$. The group-like elements of
$\mathfrak{b}^*$ are exactly the algebra homomorphisms
$\textrm{Alg}_{\mathbb{K}}(\mathfrak{b},\mathbb{K})$, so to verify
that a particular homomorphism is the distinguished group-like
element, it suffices to compute its  values on the generators.

\begin{prop}\label{6.4} Write $2\rho=10\alpha_1+6\alpha_2$, where $\rho$ is the half sum of
positive roots. Let $\gamma \in
\textrm{Alg}_\mathbb{K}(\mathfrak{b},\mathbb{K})$ be defined by
$$ \gamma(e_k)=0\quad \textrm{and} \quad
\gamma(\omega_k)=\langle
\omega_1',\omega_k\rangle^{10}\langle\omega_2',\omega_k\rangle^6.
\leqno (6.1)$$ Then $\gamma$ is the distinguished group-like element
of $\mathfrak{b}^*$.
\end{prop}
\begin{proof}
It suffices to argue that $\gamma$ as in (6.1) satisfies
$ya=\gamma(a)y$ for $a=e_k$ and $a=\omega_k$, $1\leq k\leq 2$, and
for $y=tx$ given in Proposition 6.1. Recall from the proof of
Proposition 6.2 that $xe_k=0$. Thus, $ye_k=txe_k=0=\gamma(e_k)y$. We
have
\begin{gather*}
\begin{split}
y\omega_k& =tx\omega_k=t(E_{2}^{\ell-1}E_{12}^{\ell-1}E_{11212}^{\ell-1}
E_{112}^{\ell-1}E_{1112}^{\ell-1}E_{1}^{\ell-1})\omega_k\\
&=\langle
\omega_1',\omega_k\rangle^{-10(\ell-1)}\langle\omega_2',\omega_k\rangle^{-6(\ell-1)}t\omega_kx=\langle
\omega_1',\omega_k\rangle^{10}\langle\omega_2',\omega_k\rangle^6t\omega_kx=
\gamma(\omega_k)y
\end{split}
\end{gather*} for $k=1,2$.
\end{proof}
Under some assumptions of Lemma \ref{5.1},
$(\mathfrak{b}')^{coop}\cong \mathfrak{b}^*$ as Hopf algebras, via
the map $\phi:(\mathfrak{b}')^{coop}\longrightarrow \mathfrak{b}^*,
\phi(\omega'_j)=\gamma_j, \phi(f_j)=\eta_j$ (The $\gamma_j$ and
$\eta_j$ are defined in the proof of Lemma \ref{5.1}). This allows
us to define a Hopf pairing
$\mathfrak{b}'\times\mathfrak{b}\rightarrow \mathbb{K}$ whose values
on generators are given by
$$(f_j\mid e_i)=\delta_{ij},\quad (\omega'_j\mid \omega_i)=\langle \omega'_j,\omega_i\rangle,\leqno
(6.2)$$ and on all other pairs of generators are $0$. If we set
$\omega'_{2\rho}:={\omega'_{1}}^{10}{\omega'_2}^{6}$, then
$(\omega'_{2\rho}\,|\, b)=\gamma(b)$ for all $b \in \mathfrak{b}$,
that is, $\gamma=(\omega_{2\rho}'\,|\,\cdot)$.

Note that $\mathfrak b_{s^{-1},r^{-1}}\cong (\mathfrak
b')^{\textit{coop}}\cong \mathfrak b^*$ as Hopf algebras. Under the
isomorphism $\phi\psi^{-1}$ (where $\psi(f_i)=e_i$,
$\psi(\omega_i')=\omega_i$), a nonzero left (resp., right) integral
of $\mathfrak{b}$ maps to a nonzero left (resp., right) integral of
$\mathfrak{b}^*$. Thus, we have
\begin{prop}
Let $\lambda =\nu\eta$ and $\lambda'=\eta\nu \in \mathfrak{b}^*$,
where
$$
\nu=\prod_{i=1}^{2}(1+\gamma_i+\cdots+\gamma_i^{\ell-1}),\quad
\eta=\eta_{2}^{\ell-1}\eta_{12}^{\ell-1}\eta_{11212}^{\ell-1}\eta_{112}^{\ell-1}\eta_{1112}^{\ell-1}\eta_{1}^{\ell-1},
$$
here $\eta_{12}=[\eta_1,\eta_{2}]_{r^3},$
$\eta_{112}=[\eta_{1},\eta_{12}]_{r^2s},$
$\eta_{1112}=[\eta_{1},\eta_{112}]_{rs^2},$
$\eta_{11212}=[\eta_{112},\eta_{12}]_{rs^2}$. Then $\lambda$ is a
left integral and $\lambda'$ is a right integral of
$\mathfrak{b}^*$. \hfill\qed
\end{prop}

\begin{prop}\label{6.6}
$g:=\omega_{2\rho}^{-1}$ is the distinguished group-like element of
$\mathfrak{b}$, and under the Hopf pairing in $(6.2)$,
$(\omega'_i\mid g)=\langle
\omega_i',\omega_1\rangle^{-10}\langle\omega_i',\omega_2\rangle^{-6}=
\gamma_i(g)$ for $i=1,2$.
\end{prop}
\begin{proof} Let
$F=F_{2}^{\ell-1}F_{12}^{\ell-1}F_{11212}^{\ell-1}F_{112}^{\ell-1}F_{1112}^{\ell-1}F_{1}^{\ell-1}$.
Then we have
$$
\omega'_k F=\langle
\omega_k',\omega_1\rangle^{-10}\langle\omega_k',\omega_2\rangle^{-6}F
\omega'_k.
$$
Since
$\phi^{-1}(\lambda')=F\Bigl(\prod_{i=1}^{2}(1+\omega'_i+\cdots+(\omega'_i)^{\ell-1})\Bigr),$
it follows that
$$\gamma_k \lambda'=\langle
\omega_k',\omega_1\rangle^{-10}\langle\omega_k',\omega_2\rangle^{-6}\lambda',$$
and $\eta_k \lambda'=0$. Taking $g:=\omega_{2\rho}^{-1}$, we have
$\xi\lambda'=\xi(g)\lambda'$ for all $\xi \in \mathfrak{b}^*$.
\end{proof}

\begin{theorem}
Assume that $r,s$ are $\ell$th roots of unity. Then
$D(\mathfrak{b})$ has a ribbon element.
\end{theorem}
\begin{proof} By Proposition \ref{6.6}, $g=\omega_{2\rho}^{-1}$ is the distinguished
group-like element of $\mathfrak{b}$. There exists a group-like
element $h=\omega_{\rho}^{-1}\in \mathfrak{b}$ such that $h^2=g$.
Because $\gamma=(\omega_{2\rho}'\,|\,\cdot)$ corresponds to
$\omega'_{2\rho}$ under the isomorphism $\phi^{-1}:
\mathfrak{b}^*\longrightarrow (\mathfrak{b}')^{coop}$,  there exists
a $\delta=(\omega_\rho'\,|\,\cdot) \in \mathfrak{b}^*$ such that
$\delta^2=\gamma$, which is given by
$$ \delta(e_k)=0,\quad \textit{and} \quad
\delta(\omega_k)=\langle\omega_1',\omega_k\rangle^5\langle\omega_2',\omega_k\rangle^3,\qquad
k=1,2.$$ Then
\begin{gather*}
\begin{split}
h(\delta\rightharpoonup\omega_k\leftharpoonup\delta^{-1})h^{-1}&=\delta(\omega_k)
\delta^{-1}(\omega_k)h\omega_kh^{-1}=\omega_k=S^2(\omega_k),\\
h(\delta\rightharpoonup
e_k\leftharpoonup\delta^{-1})h^{-1}&=\delta(1)
\delta^{-1}(\omega_k)he_kh^{-1}\\
&=\langle\omega_1',\omega_k\rangle^{-5}\langle\omega_2',\omega_k\rangle^{-3}he_kh^{-1}\\
&=\langle\omega_1',\omega_k\rangle^{-5}\langle\omega_2',\omega_k\rangle^{-3}\langle\omega_k',\omega_1\rangle^{-5}\langle\omega_k',\omega_2\rangle^{-3}e_k\\
&=\omega_k^{-1}e_k\omega_k=S^2(e_k).
\end{split}
\end{gather*}

Then Theorem 6.1 implies the result.
\end{proof}

Under the hypothesis of Theorem \ref{5.2},
$\mathfrak{u}_{r,s}(G_2)\cong D(\mathfrak{b})$. Thus, we have
\begin{coro}
Assume that $r=\theta^y$, $s=\theta^z$, where $\theta$ is a
primitive $\ell$th root of unity and $(3(y^{2}+z^2+yz),\ell)=1$.
Then $\mathfrak{u}_{r,s}(G_2)$ has a ribbon element. \hfill\qed
\end{coro}

\bigskip
\bibliographystyle{amsalpha}

\begin{thebibliography}{A}
\medskip
\bibitem[AS1]{AS1} N. Andruskiewitsch and H.J. Schneider, \textit{Lifting of quantum linear spaces and pointed
Hopf algebras of order $p^3$}, J. Algebra  \textbf{209} (1998),
658--691.

\bibitem[AS2]{AS2} N. Andruskiewitsch and H.J. Schneider, \textit{Pointed Hopf algebras},
New Directions in Hopf Algebras, pp.1--68, Math. Sci. Res. Inst.
Publ., \textbf{43}, Cambridge Univ. Press, Cambridge, 2002.

\bibitem[AS3]{AS3} N. Andruskiewitsch and H.J. Schneider, \textit{On the classification of finite-dimensional
pointed Hopf algebras}, to appear in Ann. Math., arXiv math.
QA/0502157.

\bibitem[AS4]{AS4} N. Andruskiewitsch, \textit{About finite dimensional
Hopf algebras}, Contemp. Math., \textbf{294} (2002), 1--57.

\bibitem[BDG]{BDG} M. Beattie, S. Dascalescu and L. Gr\"unenfelder, \textit{On the number of types of finite
dimensional Hopf algebras}, Invent. Math. \textbf{136} (1999), 1--7.

\bibitem[B]{B} J. Beck, \textit{Convex bases of PBW type for quantum affine algebras}, Comm. Math. Phys. \textbf{165} (1994), 193--199.

\bibitem[Be]{Be} G. Benkart, \textit{Down-up algebras and Witten's deformations of the universal enveloping algebra of
$\mathfrak{sl}_2$}, in ``Recent Progress in Algebra" ed. S.G. Hahn,
H.C. Myung, and E. Zemanov, Contemp. Math. \textbf{224}, Amer. Math.
Soc., Province, RI (1999), 29--45.

\bibitem [BGH1]{BGH1} N. Bergeron, Y. Gao and N. Hu, \textit{Drinfel'd doubles and
Lusztig's symmetries of two-parameter quantum groups}, J. Algebra
\textbf{301} (2006), 378--405.

\bibitem [BGH2]{BGH2} N. Bergeron, Y. Gao and N. Hu, \textit{Representations of
two-parameter quantum orthogonal groups and symplectic groups},
 AMS/IP, Studies in Advanced Mathematics,  vol. \textbf{39}, pp. 1--21, 2007. arXiv math. QA/0510124.

\bibitem [BH]{BH}X. Bai and N. Hu, \textit{Two-parameter quantum group of exceptional type
$E$-series  and  convex \textrm{PBW} type basis}, Algebra Colloq.
\textbf{15} (4) (2008), 619--636. arXiv. Math. QA/0605179.

\bibitem [BKL1]{BKL1} G. Benkart, S.J. Kang and K.H. Lee, \textit{PBW-type bases of two-parameter quantum groups (of type
$A$)}, preprint, 2003.

\bibitem [BKL2]{BKL2} G. Benkart, S.J. Kang and K.H. Lee, \textit{On the center of two-parameter quantum groups (of type $A$)},
Proc. Roy. Soc. Edingburg Sect. A, \textbf{136} (3) (2006),
445--472.

\bibitem [BW1]{BW1} G. Benkart and S. Witherspoon, \textit{Two-parameter quantum
groups  (of type $A$) and Drinfel'd doubles}, Algebra Represent.
Theory, \textbf{7} (2004), 261--286.

\bibitem [BW2]{BW2} G. Benkart and S. Witherspoon, \textit{Representations of two-parameter quantum
groups (of type $A$) and Schur-Weyl duality}, in ``Hopf Algebras",
Lecture Notes in Pure and Appl. Math., \textbf{237}, pp. 65--92,
Dekker, New York, 2004.

\bibitem [BW3]{BW3} G. Benkart and S. Witherspoon, \textit{Restricted two-parameter quantum groups  (of type $A$)},
Fields Institute Communications, ``Representations of Finite
Dimensional Algebras and Related Topics in Lie Theory and Geometry",
vol. \textbf{40}, pp. 293--318, Amer. Math. Soc., Providence, RI,
2004.


\bibitem[DK]{DK} C. De Concini and V. Kac, \textit{Representations of
quantum groups at a root of $1$}, in ``Operator Algebras, Unitary
Representations, Enveloping Algebras, and Invariant Theory (Paris,
1989)," A. Connes, M. Duflo, A. Joseph, and R. Rentschler eds.,
Progress in Math. \textbf{92}, pp. 471--506, Birkh\"auser Boston,
1990.

\bibitem [Dr]{Dr} V.G. Drinfel'd, \textit{Quantum groups}, in
``Proceedings ICM'' (Berkeley 1986), pp. 798--820,  Amer. Math. Soc.
1987.

\bibitem[EG]{EG} P. Etingof and S. Gelaki, \textit{The
classification of triangular semisimple and cosemisimple Hopf
algebras over an algebraically closed field}, Internat. Math. Res.
Notices \textbf{5} (2000), 223--234.

\bibitem[G]{G} S. Gelaki, \textit{On pointed Hopf algebras and Kaplansky's 10th conjecture}, J.
Algebra \textbf{209} (1998), 635--657.


\bibitem[H1]{H1} N. Hu, \textit{Quantum divided power algebra,
$q$-derivatives, and some new quantum groups}, J. Algebra,
\textbf{232} (2000), 507--540.

\bibitem[H2]{H2} N. Hu, \textit{Quantum group structure associated to the quantum affine space}, Algebra Colloq.
\textbf{11} (2004), 483--492. (Pr\'epublication de l'Institut de
Recherche Math\'ematique Avanc\'ee, Nr. dans la collection: 026,
(2001), ISSN Collection: 0755-3390.)


\bibitem [H3]{H3} N. Hu, \textit{Lyndon words, convex PBW bases and their R-matrices for the
two-parameter quantum groups of $B_2$, $C_2$, $D_4$ types},
manuscript 2005.

\bibitem [HRZ]{HRZ} N. Hu, M. Rosso and H. Zhang, \textit{Two-parameter affine quantum group
$U_{r,s}(\widehat{\mathfrak{sl}_n})$, Drinfel'd realization and
quantum affine Lyndon basis}, Comm. Math. Phys. \textbf{278} (2)
(2008), 453--486.

\bibitem [HS]{HS} N. Hu and Q. Shi, \textit{The two-parameter quantum group of exceptional type $G_2$
and Lusztig symmetries}, Pacific J. Math., \textbf{230} (2),
327--346.

\bibitem[HW1]{HW1} N. Hu and X. Wang, \textit{Quantizations of the generalized-Witt algebra and of Jacobson-Witt
 algebra in the modular case}, arXiv Math: QA/0602281, J. Algebra, \textbf{312} (2007),
 902--929.

\bibitem[HW2]{HW2} N. Hu and X. Wang, \textit{Convex PBW-type Lyndon basis and restricted two-parameter
 quantum groups of type $B$}, preprint, 2007.


\bibitem[HZ]{HZ} N. Hu and H. Zhang, \textit{Vertex representations of two-parameter quantum affine algebras
$U_{r,s}(\widehat{\mathfrak{g}})$: the simply-laced cases}, Preprint
2006-2007.

\bibitem [Jim]{Jim} M. Jimbo, \textit{A q-analog of $U(gl_{N{+}1})$, Hecke algebra,
and the Yang-Baxter equation}, Lett. Math. Phys, \textbf{11} (1986),
247-252.

\bibitem [K]{K} I. Kaplansky, \textit{Bialgebras}, University of Chicago Lecture Notes, Chicago, 1975.

\bibitem [Ka]{Ka} C. Kassel, \textit{Quantum Groups}, GTM
\textbf{155}, Springer-Verlag Berlin/Heidelberg/New York, 1995.

\bibitem [KR] {KR} L.H. Kauffman and D.E. Radford, \textit{A necessary and
sufficient condition for a finite-dimensional Drinfel'd double to be
a ribbon Hopf algebra}, J. Algebra \textbf{159} (1993), 98--114.

\bibitem[K1]{K1}V.K. Kharchenko, {\it A quantum analog of the
Poincar$\acute{e}$-Birkhoff-Witt theorem}, Algebra and Logic, {\bf
38} (1999), 259--276.

\bibitem [K2]{K2} V.K. Kharchenko, \textit{A combinatorial approach to
the quantification of Lie algebras}, Pacific J. Math.  \textbf{23}
(2002), 191--233.

\bibitem [L1]{L1} G. Lusztig, \textit{Finite-dimensional Hopf
algebras arising from quantized universal enveloping algebras}, J.
Amer. Math. Soc. \textbf{3} (1990), 257--296.

\bibitem [L2]{L2} G. Lusztig, \textit{Quantum groups at roots of $1$}, Geom. Dedicata \textbf{35}
(1990), 89--114.

\bibitem [L3]{L3} G. Lusztig, \textit{Introduction to Quantum Groups}, Birkh\"auser Boston, 1993.

\bibitem[LR]{LR} M. Lalonde and A. Ram, \textit{Standard Lyndon bases of Lie algebras and enveloping algebras}, Trans. Amer. Math. Soc.,
\textbf{347} (5) (1995), 1821--1830.

\bibitem [M]{M} S. Montgomery, \textit{Hopf Algebras and Their Actions on Rings}, CBMS Conf.
Math. Publ., \textbf{82}, Amer. Math. Soc., Providence, 1993.

\bibitem [M1]{M1} S. Montgomery, \textit{Classifying finite dimensional semsimple Hopf Algebras},
Contemp. Math., \textbf{229} (1998), AMS, 265-279.

\bibitem [R]{R} D.E. Radford, \textit{The order of the antipode of
 a finite-dimensional Hopf algebra is finite}, Amer. J. Math.
 \textbf{98} (1976), 333--355.

\bibitem[R1]{R1} M. Rosso, \textit{Quantum groups and quantum
shuffles,} Invent. Math. \textbf{133} (2) (1998), 399--416.

\bibitem[R2]{R2}M. Rosso,  \textit{Lyndon bases and the multiplicative formula for $R$-matrices},
   (2002),  preprint.

\bibitem [RT1]{RT1} N.Yu. Reshetikhin and V.G. Turaev,
\textit{Ribbon graphs and their invariants derived from quantum
groups}, Comm. Math. Phys. \textbf{127} (1990), 1--26.

\bibitem [RT2]{RT2} N.Yu. Reshetikhin and V.G. Turaev,
\textit{Invariants of $3$-manifolds via link polynomials and quantum
groups}, Invent. Math. \textbf{103} (1991), 547--597.

\bibitem [RTF]{RTF} N.Yu. Reshetikhin, L.A. Takhtajan and L.D. Faddeev,
\textit{Quantization of Lie groups and Lie algebras}, Algebra and
Anal. \textbf{1}, 178--206 (1989) (Leningrad Math. J. 1
     [Engl. transl. 193--225 (1990)]).

\bibitem [S]{S} M.E. Sweedler,
\textit{Hopf Algebras}, Mathematics Lecture Note Series W. A.
Benjamin, New York, 1969.

\bibitem[T]{T} E.J. Taft, \textit{The order of the antipode of
finite-dimensional Hopf algebra}, Proc. Nat. Acad. Sci. USA.
\textbf{68} (1971), 2631--2633.


\end{thebibliography}

\end{document}